\documentclass[10pt,a4paper]{article}
\usepackage{fullpage,mathrsfs,graphicx,framed,color,amssymb,amsmath,amsthm}
\usepackage[boxed, noline, ruled, linesnumbered]{algorithm2e}
\usepackage[colorlinks, citecolor=blue]{hyperref}
\usepackage{epsfig, graphicx}
\usepackage{latexsym,amsfonts,amsbsy,amssymb}
\usepackage{amsmath,amsthm}
\usepackage{enumerate}
\makeatletter %'@' is now a normal "letter" for TeX

\@addtoreset{equation}{section}
\makeatother %'@' is restored as a "non-letter" character for TeX
\allowdisplaybreaks % For long formula
\newtheorem{theorem}{Theorem}[section]
\newtheorem{lemma}{Lemma}[section]
\newtheorem{corollary}{Corollary}[section]
\newtheorem{remark}{Remark}[section]
\newcommand{\comm}[1]{{\color{black}#1}}
\newcommand{\revise}[1]{{\color{black}#1}}
\begin{document}
\title{A Nonnested Augmented Subspace Method for Eigenvalue Problems with Curved Interfaces\footnote{This work was supported in part
by the National Key Research and Development Program of China (2019YFA0709601), Science Challenge Project (No. TZ2016002),
 National Natural Science Foundations of China (NSFC 11771434,  91730302,
91630201), the National Center for Mathematics and Interdisciplinary Science, CAS.}}
\author{Haikun Dang\footnote{ICMSEC, LSEC, NCMIS,
Academy of Mathematics and Systems Science, Chinese Academy of
Sciences, Beijing 100190, China, and School of Mathematical
Sciences, University of Chinese Academy of Sciences, Beijing,
100049, China ({\tt danghaikun@lsec.cc.ac.cn}).},\ \ \
Hehu Xie\footnote{ICMSEC, LSEC, NCMIS,
Academy of Mathematics and Systems Science, Chinese Academy of
Sciences, Beijing 100190, China, and School of Mathematical
Sciences, University of Chinese Academy of Sciences, Beijing,
100049, China ({\tt hhxie@lsec.cc.ac.cn}).}, \ \
Gang Zhao\footnote{ICMSEC, LSEC, NCMIS,
Academy of Mathematics and Systems Science, Chinese Academy of
Sciences, Beijing 100190, China, and School of Mathematical
Sciences, University of Chinese Academy of Sciences, Beijing,
100049, China ({\tt zhaog6@lsec.cc.ac.cn}).} \ \ \ and \ \
Chenguang Zhou\footnote{ICMSEC, LSEC, NCMIS,
Academy of Mathematics and Systems Science, Chinese Academy of
Sciences, Beijing 100190, China, and School of Mathematical
Sciences, University of Chinese Academy of Sciences, Beijing,
100049, China ({\tt zhouchenguang@lsec.cc.ac.cn}).}}
\date{} \maketitle

\begin{abstract}
%% Text of abstract
In this paper, we present a nonnested augmented subspace algorithm and its multilevel correction method
for solving eigenvalue problems with curved interfaces.
The augmented subspace algorithm and the corresponding multilevel correction method
are designed based on a coarse finite element space which is not the subset of the finer finite element space.
The nonnested augmented subspace method can transform the eigenvalue problem solving on the finest mesh
to the solving linear equation on the same mesh and small scale eigenvalue problem on the low dimensional augmented subspace.
The corresponding theoretical analysis and numerical experiments are provided to demonstrate
the efficiency of the proposed algorithms.

\vskip0.3cm {\bf Keywords.} Nonnested augmented subspace method,  multilevel correction method,
finite element method, eigenvalue problem, curved interface.

\vskip0.2cm {\bf AMS subject classifications.} 65N30, 65N25, 65L15, 65B99.
\end{abstract}

%% main text
\section{Introduction}\label{intro}
There exist a lot of eigenvalue problems in scientific research and practical engineering.
\revise{Especially, along with the development of modern science and technology,
the scale of eigenvalue problems is becoming larger and larger, which leads to the urgent demand for efficient numerical
methods for eigenvalue problems. It is well known that multigrid methods have been developed to be very mature and
produced an almost complete set of solvers and theoretical systems for solving linear boundary value problems
\cite{BankDupont,Bramble,BramblePasciak,BrambleZhang,BrandtMcCormickRuge,Hackbusch_Book,ScottZhang,shaidurov1995multigrid,ToselliWidlund,Xu,
Xu_Nonsymmetric,XuXiaowen}. On the contrary, the applications of the multigrid methods to solving nonlinear problems and eigenvalue problems
are very few and need more attentions.
In order to use the multigrid method, the normal way is to linearize the nonlinear problems
with some type of nonlinear iteration.
Then we solve the linearized  equations with the help of multigrid methods.
This is always called the outer iteration (nonlinear iteration) plus the inner iteration (multigrid iteration).}
Although the multigrid method has the best efficiency for the inner iteration,
the total computational work is controlled by the number of outer iteration steps.
When the concerned problem has strong nonlinearity and needs many outer iteration steps,
the computational work will be very large even though the multigrid method is used for the inner iteration.
Based on this understanding, the application of multigrid algorithms does not affect the outer iteration
and can not make the total computational work be independent of the nonlinear iterations.

A special example among nonlinear equations is the eigenvalue problem which originates from applied mathematics,
physics, chemistry, cybernetics and other disciplines.
Similarly, the multigrid algorithms for eigenvalue problems have not been developed so well, even there exist some
numerical methods from Hackbusch \cite{Hackbusch}, Brandt \cite{BrandtMcCormickRuge},
Shaidurov \cite{shaidurov1995multigrid} and so on.
Since these multigrid methods are designed based on inverse power method or Rayleigh quotient iteration,
we always need to solve almost singular linear equations during the whole process.
For this reason, the corresponding computational work depends on eigenvalue distributions.
\revise{It is more difficult to design some type of numerical methods for solving the eigenvalue problems
with the optimal computational complexity and storage as that for the linear boundary value problems.}
From this point of view, the application of multigrid method
does not leads to a new eigensolver.

\revise{
In recent years, multilevel correction methods and their corresponding multigrid algorithms
for eigenvalue problems and nonlinear problems have been proposed and discussed in \cite{ChenHeLiXie,ChenXieXue,GongXieYan_SIAM,GongXieYan_JSC,hanYangBi,HanLiXie,HanLiXieYou,HanXieXu,HuXieXu,JiSunXie,JXXX,LinXie_2010,
LinXie_2012,LinXie,LXX,PengBiLiyang,XiJiZhang,Xie_IMA,Xie_JCP,Xie_Nonlinear,Xie_BIT,XieWu,
XieXie,XieXieZhang,XieZhangOwhadi,XieZhou,XuXie,YueXieXie,ZhangXuXie,ZhangHanHeXieYou,ZhangXiJi}.
This type of multilevel correction methods can transform the eigenvalue problem solving
into solving standard linear equations and eigenvalue problems in a very low dimensional space.
This process makes the computational work for solving the eigenvalue problems be equivalent to that for solving
the corresponding linear problems by adjusting the low-dimensional spaces.}
Among these existing multilevel correction and multigrid methods, the concerned sequence of meshes are required to be
nested which means the finite element space defined on the coarse mesh is a subset of the one defined on the finer meshes.
This standard requirement forbids the applications of multilevel correction methods in the adaptive triangulations
which are generated by moving meshes \cite{DiLiTangZhang,LiTangZhang_1,LiTangZhang_2,Miller_2,Miller_1}.
For example, when the eigenvalue problem is defined on the domain with curved interfaces and piecewise constant coefficients,
in order to guarantee the approximation accuracy for the curve interface, we can not produce
the nested coarse and finer meshes for the multilevel correction method.
The aim of this paper is to propose a type of nonnested augmented subspace method and \comm{then multilevel correction scheme}
for solving the eigenvalue problems with curved interfaces and piecewise constant coefficients.

An outline of this paper goes as follows. In Section \ref{FEMEP}, we introduce the finite element method for
the eigenvalue problem and the corresponding error estimate theory. A nonnested augmented subspace method
for the eigenvalue problem is proposed in Section \ref{ASA}. %, and the error estimates are derived.
In Section \ref{MLCM}, we design a type of multilevel correction method for the eigenvalue problem
based on the augmented subspace method in Section \ref{ASA}.  %and provide the results of error analysis.
In Section \ref{NE}, four numerical examples are provided
to validate the theoretical results and illustrate the efficiency of proposed algorithms in this paper.
Finally, some concluding remarks are given in the last section.

%--------------------------------------------------------------------------------------------------

\section{Finite element method of the eigenvalue problem}\label{FEMEP}
This section is devoted to introducing some notation and the standard finite element
method for the eigenvalue problem. In this paper, we shall use the standard notation
for Sobolev spaces $W^{s,p}(\Omega)$ and their
associated norms and semi-norms (cf. \cite{Adams}). For $p=2$, we denote
$H^s(\Omega)=W^{s,2}(\Omega)$ and
$H_0^1(\Omega)=\{v\in H^1(\Omega):\ v|_{\partial\Omega}=0\}$,
where $v|_{\partial\Omega}=0$ is in the sense of trace,
$\|\cdot\|_{s,\Omega}=\|\cdot\|_{s,2,\Omega}$.
In some places, $\|\cdot\|_{s,2,\Omega}$ should be viewed as piecewise
defined if it is necessary.
The letter $C$ (with or without subscripts) denotes a generic
positive constant which may be different at
 its different occurrences throughout the paper.

In this paper, we are concerned with the following second order elliptic eigenvalue problem: Find $(\lambda, u)$ such that
\begin{equation}\label{LaplaceEigenProblem}
\left\{
\begin{array}{rcl}
-\nabla\cdot(\mathcal{A}\nabla u)&=&\lambda u, \ \quad \text{in} \  \Omega,\\
{[}u{]}=0, \ \ {\big[}\mathbf n_\Gamma \mathcal A\nabla u{\big]}&=&0, \ \ \ \quad \text{on}\  \Gamma,\\
u&=&0, \ \ \ \quad \text{on}\  \partial\Omega,
\end{array}
\right.
\end{equation}
where the computing domain $\Omega$ has curved interfaces $\Gamma$ which denotes the set of all involved interfaces,
$\mathcal{A}=(a_{i,j})_{d\times d}$ is a symmetric positive definite matrix
and $a_{i,j}\in W^{0,\infty}(\Omega) \ (i,j=1,2,\cdots ,d)$ are piecewise constants.
In this paper, $[v] := (v|_{\Omega_i} )|_{\Gamma}-(v|_{\Omega_j})|_{\Gamma}$ for any  function $v$ in $H^1(\Omega)$,
where $\Omega_i$ and $\Omega_j$ are any two adjacent subdomains and
$\mathbf n_\Gamma$ denotes a unit normal vector from $\Omega_i$ to $\Omega_j$ across the interface.
Figure \ref{fig_domain} shows an example of computing domain with four curved interfaces.
\begin{figure}[!htpb]
\centering
\includegraphics[width=7cm,height=5cm]{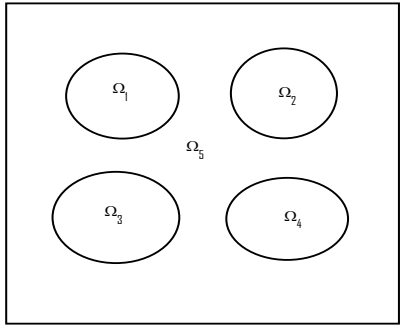}
\caption{Domain with curved interfaces}\label{fig_domain}
\end{figure}

In order to use the finite element method to solve
the eigenvalue problem (\ref{LaplaceEigenProblem}), we  define
the corresponding variational form as follows:
Find $(\lambda, u )\in \mathcal{R}\times V$ such that $a(u,u)=1$ and
\begin{eqnarray}\label{weak_eigenvalue_problem}
a(u,v)=\lambda b(u,v),\quad \forall v\in V,
\end{eqnarray}
where $V:=H_0^1(\Omega)$,  and the  bilinear forms $a(\cdot,\cdot)$ and $b(\cdot,\cdot)$ are defined as
\begin{equation}\label{inner_product_a_b}
a(u,v)=\int_{\Omega}\mathcal{A}\nabla u\cdot\nabla v d\Omega,
 \ \ \ \  \ \ b(u,v) = \int_{\Omega}uv d\Omega.
\end{equation}
The norms $\|\cdot\|_a$ and $\|\cdot\|_b$ are defined by
\begin{eqnarray*}
\|v\|_a=\sqrt{a(v,v)}\ \ \ \ \ {\rm and}\ \ \ \ \ \|v\|_b=\sqrt{b(v,v)}.
\end{eqnarray*}
It is easy to known that $a(u,v)$ satisfies boundedness and coercive property on $V$, i.e.,
\begin{eqnarray}\label{coercive}
a(u,v)\leq C_a\|u\|_{1,\Omega}\|v\|_{1,\Omega}\ \
\text{and} \ \ c_a\|u\|_{1,\Omega}^2\leq a(u,u), \quad \forall\ u,v\in V.
\end{eqnarray}
Then the norm $\|\cdot\|_a$ is equivalent to the one $\|\cdot\|_1$.

It is standard that the eigenvalue problem (\ref{weak_eigenvalue_problem})
has an eigenvalue sequence $\{\lambda_j \}$ (cf. \cite{BabuskaOsborn_1989,Chatelin}):
$$0<\lambda_1 < \lambda_2\leq\cdots\leq\lambda_k\leq\cdots,\ \ \
\lim_{k\rightarrow\infty}\lambda_k=\infty,$$
and the associated eigenfunctions
$$u_1, u_2, \cdots, u_k, \cdots,$$
where $a(u_i,u_j)=\delta_{ij}$ ($\delta_{ij}$ is the Kronecker function).
In the sequence $\{\lambda_j\}$, the $\lambda_j$ are repeated based on their geometric multiplicity.

For the theoretical analysis in this paper, we present the definition corresponding to the smallest eigenvalue $\lambda_1$
(c.f. \cite{BabuskaOsborn_1989,Chatelin}) as follows
\begin{eqnarray}\label{Smallest_Eigenvalue}
\lambda_1 = \min_{0\neq w\in V}\frac{a(w,w)}{b(w,w)}.
\end{eqnarray}

Now, we come to introduce the finite element method for (\ref{weak_eigenvalue_problem}).
First, let us define the finite element space.
Let $\mathcal{T}_h$ be a regular partition of $\Omega\subset \mathbb R^d \ (d=2, 3)$ which means
a two-dimensional domain is divided into regular triangles or quadrangles
(a three-dimensional domain is divided into tetrahedrons or hexahedrons) \cite{BrennerScott, Ciarlet}.
Denote the diameter of a element $K \in \mathcal{T}_h$ by $h_K$, and $h$ describes
the maximum diameter of all elements of $\mathcal{T}_h$.
In order to guarantee the accuracy of finite element spaces, the domain is usually \revise{partitioned along the interior edges} (or faces)
so that the partition has a certain approximating accuracy to the curved interfaces.
Based on the mesh $\mathcal{T}_h$, we can construct a finite element space denoted by
 $V_h \subset V$. For simplicity, we set $V_h$ as the linear finite
 element space which is defined as follows
\begin{equation}\label{linear_fe_space}
  V_h = \big\{ v_h \in C(\Omega)\ \big|\ v_h|_{K} \in \mathcal{P}_1,
  \ \ \forall K \in \mathcal{T}_h\big\}\cap H^1_0(\Omega),
\end{equation}
where $\mathcal{P}_1$ denotes the linear function space.
\revise{Since the appearance of curved interfaces and the accuracy requirement,
there is no  nested sequence of meshes as that for the polygonal domains.
Then we have no nested sequence of finite element spaces which is always needed in the multigrid method.}

Based on the  space $V_h$, the standard finite element scheme for eigenvalue
problem (\ref{weak_eigenvalue_problem}) is:
Find $(\bar{\lambda}_h, \bar{u}_h)\in \mathcal{R}\times V_h$
such that $a(\bar{u}_h,\bar{u}_h)=1$ and
\begin{eqnarray}\label{Weak_Eigenvalue_Discrete}
a(\bar{u}_h,v_h)=\bar{\lambda}_h b(\bar{u}_h,v_h),\quad\ \  \ \forall\ v_h\in V_h.
\end{eqnarray}
From \cite{BabuskaOsborn_1989,Chatelin}, we know that the discrete eigenvalue problem (\ref{Weak_Eigenvalue_Discrete}) has an eigenvalue sequence
$$0<\bar{\lambda}_{1,h}\leq \bar{\lambda}_{2,h}\leq\cdots\leq \bar{\lambda}_{k,h}
\leq\cdots\leq \bar{\lambda}_{N_h,h},$$
and the corresponding discrete eigenfunction sequence
$$\bar{u}_{1,h}, \bar{u}_{2,h}, \cdots, \bar{u}_{k,h}, \cdots, \bar{u}_{N_h,h},$$
where $a(\bar{u}_{i,h},\bar{u}_{j,h})=\delta_{ij}$, $1\leq i,j\leq N_h$ ($N_h=\textrm{dim}V_h$).\\

In order to measure the error of the finite element space to the desired function, we define the following notation
\begin{eqnarray}\label{Delta_V_h}
\delta(w,V_h) = \inf_{v_h\in V_h}\|w-v_h\|_a,\ \ \ {\rm for}\ w\in V.
\end{eqnarray}
In this paper, we also need the following quantity for error analysis:
\revise{
\begin{eqnarray}\label{Definition_Eta_a_h}
\eta_a(V_h)&=&\sup_{\substack{ f\in L^2(\Omega)\\ \|f\|_b=1}}\inf_{v_h\in V_h}\|Tf-v_h\|_{a},\label{eta_a_h_Def}
\end{eqnarray}
}
where $T:L^2(\Omega)\rightarrow V$ is defined as
\begin{equation}\label{laplace_source_operator}
\revise{a(Tf,v) = b(f,v), \ \ \ \ \  \forall v\in V\ \ {\rm for}\  f \in L^2(\Omega).}
\end{equation}
It is known that $\eta_a(V_h)\rightarrow 0$ when $h\rightarrow 0$ (c.f. \cite{Babuska_1970,ChenZou,Li201019}).
Based on the finite element space $V_h$, we define the finite element projection operator
$\mathcal P_h: V\rightarrow V_h$ as follows
\begin{eqnarray}\label{Definition_FEM_Projection}
a(w, v_h)=a(\mathcal P_hw, v_h),\ \ \ \ \forall\ v_h\in V_h, \ \textrm{for}\ w\in V.
\end{eqnarray}
It is obvious that $\delta(u,V_h) = \|u-\mathcal P_hu\|_a$.

In order to introduce and analyze the nonnested augmented subspace algorithm and the corresponding
multilevel correction method for the eigenvalue problem, we state the following error estimate results from \cite{XieZhangOwhadi}
which include only explicit constants. For more details, please refer to \cite{XieZhangOwhadi}.

It should be pointed out that the following error estimate results hold for general finite-dimensional approximations
of eigenvalue problems.
\begin{lemma}\label{Error_Estimate_Theorem}(\cite{XieZhangOwhadi})
Let $(\lambda,u)$ be an exact eigenpair of the eigenvalue problem (\ref{weak_eigenvalue_problem}).
Assume the eigenpair approximation $(\bar\lambda_{i,h},\bar u_{i,h})$ has the property that
$\bar\mu_{i,h}=1/\bar\lambda_{i,h}$ is the closest to $\mu=1/\lambda$.
The corresponding spectral projection operators $E_{i,h}: V\mapsto{\rm span}\{\bar u_{i,h}\}$
and $E: V\mapsto {\rm span}\{u\}$ are defined as follows
\begin{eqnarray*}
&&a(E_{i,h}w,\bar u_{i,h}) = a(w,\bar u_{i,h}),\ \ \  {\rm for}\  w\in V,\\
&&a(Ew, u) = a(w, u),\ \ \ \ \ \ \ \ \ \ \ \ {\rm for}\  w\in V.
\end{eqnarray*}
The finite element approximation $\bar u_{i,h}$ has the following error estimate
\begin{eqnarray}\label{Energy_Error_Estimate}
\|u-E_{i,h}u\|_a&\leq& \sqrt{1+\frac{\bar\mu_{1,h}}{\delta_{\lambda,h}^2}\eta_a^2(V_h)}\delta(u, V_h),
\end{eqnarray}
where $\eta_a(V_h)$ is defined from (\ref{Definition_Eta_a_h}) and $\delta_{\lambda,h}$ is defined as
\begin{eqnarray}\label{Definition_Delta}
\delta_{\lambda,h} &:=& \min_{j\neq i}|\bar\mu_{j,h}-\mu|=\min_{j\neq i}
\Big|\frac{1}{\bar\lambda_{j,h}}-\frac{1}{\lambda}\Big|.
\end{eqnarray}
Moreover, the eigenfunction approximation $\bar u_{i,h}$ has the following error estimate corresponding to $L^2$-norm
\begin{eqnarray}\label{L2_Error_Estimate}
\|u-E_{i,h}u\|_b &\leq&\Big(1+\frac{\bar\mu_{1,h}}{\delta_{\lambda,h}}\Big)\eta_a(V_h)\|u-E_{i,h}u\|_a.
\end{eqnarray}
\end{lemma}

\revise{For the convenience of analysis,  we state the following corollary
which is based on Lemma \ref{Error_Estimate_Theorem}.}

\begin{corollary}\label{Error_Superclose_Theorem}
Under the assumption of Lemma \ref{Error_Estimate_Theorem}, we have following error estimates
\begin{eqnarray}
\|\lambda u - \bar\lambda_{i,h}\bar u_{i,h}\|_b &\leq& C_\lambda
\eta_a(V_h)\|u-\bar u_{i,h}\|_a, \label{Err_Norm_0_Lambda}\\
\|u-\bar u_{i,h}\|_a &\leq& \frac{1}{1-D_\lambda\eta_a(V_h)}\delta(u,V_h),\label{Err_Norm_1_Superclose}
\end{eqnarray}
where the constants $C_\lambda$ and $D_\lambda$ are defined as
\begin{eqnarray}
&&C_\lambda = 2|\lambda|\Big(1+\frac{1}{\lambda_1\delta_{\lambda,h}}\Big)
+ \bar\lambda_{i,h}\sqrt{1+\frac{1}{\lambda_1\delta_{\lambda,h}^2}\eta_a^2(V_h)},\\
&&D_\lambda = \frac{1}{\sqrt{\lambda_1}}\left(2|\lambda|\Big(1+\frac{1}{\lambda_1\delta_{\lambda,h}}\Big)
+ \bar\lambda_{i,h} \sqrt{1+\frac{1}{\lambda_1\delta_{\lambda,h}^2}\eta_a^2(V_h)}\right).\label{Definition_D_Lambda}
\end{eqnarray}
\end{corollary}

\section{Augmented subspace algorithm}\label{ASA}
In this section, a nonnested augmented subspace method will be designed for eigenvalue problems.
With the help of the coarse space on a coarse mesh, the proposed method can transform the solution
of the eigenvalue problem to the corresponding linear boundary value
problems and eigenvalue problems on a very low dimensional augmented space.
Different from the augmented subspace or multilevel correction scheme from \cite{LinXie,Xie_JCP,Xie_IMA,XieZhangOwhadi},
the coarse space here is not the subspace of the finer finite element spaces.

In order to define the nonnested augmented subspace method, we generate a coarse mesh $\mathcal{T}_H$
with the mesh size $H$ and the coarse linear finite element space $V_H$ is
defined on the mesh $\mathcal{T}_H$.
The grids $\mathcal{T}_H$ and $\mathcal{T}_h$ have no nested properties, which results in $V_H\not\subset V_h$.
With the help of $V_H$, an augmented subspace can be designed as $V_{H,h}:= V_H +\textrm{span}\{\widetilde u_h\}$,
where $\widetilde u_h\in V_h$ denotes a finite element function defined on the finer mesh.
Although $V_H$ and $V_h$ have no nested properties, the augmented subspace $V_{H,h}$
is a finite-dimensional subspace of $V$. Therefore, we know that the error estimates in
Lemma \ref{Error_Estimate_Theorem} and Corollary \ref{Error_Superclose_Theorem} still hold for $V_{H,h}$.

Assume  we have obtained an  approximation $(\lambda_h^{(\ell)}, u_h^{(\ell)})$ for a certain
exact eigenpair. The augmented subspace iteration algorithm defined by Algorithm \ref{one correction step_Eig}
is used to improve the accuracy of $(\lambda_h^{(\ell)}, u_h^{(\ell)})$.
Here the superscript $\ell$ denotes  iteration index and
$(\lambda_h^{(\ell)}, u_h^{(\ell)})$ is the inputted eigenpair.
\begin{algorithm}[htbp]
\caption{Augmented subspace iteration algorithm}\label{one correction step_Eig}
\begin{enumerate}
\item Define the following linear boundary value problem: Find $\widehat{u}_h^{(\ell+1)}\in V_h$ such that
\begin{equation}\label{correct_source_exact_para}
a(\widehat{u}_h^{(\ell+1)},v_h) = \lambda_h^{(\ell)}b(u_h^{(\ell)},v_h),\ \  \forall\ v_h\in V_h.
\end{equation}
\revise{Solve (\ref{correct_source_exact_para}) with initial value $u_{h_{k}}^{(\ell)}$ and some algebraic multigrid steps
to obtain a new eigenfunction approximation $\widetilde{u}_h^{(\ell+1)}$
which satisfies the following estimate}
\begin{equation}\label{Contraction_Rate}
\|\widehat{u}_h^{(\ell+1)}-\widetilde{u}_h^{(\ell+1)}\|_a\leq \theta  \|\widehat{u}_h^{(\ell+1)} -u_{h_{k}}^{(\ell)}\|_a,
\end{equation}
where $\theta<1$ is independent of the mesh size $h$ and the iteration number $\ell$.

\item Define the augmented subspace $V_{H,h} = V_H +{\rm span}\{\widetilde{u}_h^{(\ell+1)}\}$
and solve the following eigenvalue problem:
Find $(\lambda_h^{(\ell+1)},u_h^{(\ell+1)})\in \mathbb R\times V_{H,h}$ such that $a(u_h^{(\ell+1)},u_h^{(\ell+1)})=1$ and
\begin{equation}\label{parallel_correct_eig_exact}
a(u_h^{(\ell+1)},v_{H,h}) = \lambda_h^{(\ell+1)}b(u_h^{(\ell+1)},v_{H,h}),
\ \ \ \ \ \forall\ v_{H,h}\in V_{H,h}.
\end{equation}
Solve (\ref{parallel_correct_eig_exact}) and the output $(\lambda_h^{(\ell+1)}, u_h^{(\ell+1)})$
is chosen such that $u_h^{(\ell+1)}$ has the largest component in ${\rm span}\{\widetilde u_h^{(\ell+1)}\}$
among all eigenfunctions of (\ref{parallel_correct_eig_exact}).
\end{enumerate}
Summarize \revise{abovementioned} two steps by defining
$$(\lambda_h^{(\ell+1)},u_h^{(\ell+1)}) = {\tt AugSubspace}(\lambda_h^{(\ell)},u_h^{(\ell)}, V_H, V_h).$$
\end{algorithm}

In order to simplify the notation, we assume the eigenvalue gap $\delta_{\lambda,h}$
has a uniform lower bound which is denoted by $\delta_\lambda$ (which can be seen as the
``true'' separation of the eigenvalue $\lambda$ from others).
This assumption is reasonable when the mesh size is small enough. We refer to
\cite[Theorem 4.6]{Saad} and Lemma \ref{Error_Estimate_Theorem} in this paper for details of the
dependence of error estimates on the eigenvalue gap.
\begin{theorem}\label{Error_Estimate_One_Smoothing_Theorem}
Assume there exists an exact eigenpair $(\lambda, u)$ such that
the eigenpair approximation $(\lambda_{h_k}^{(\ell)},u_{h_k}^{(\ell)})$ satisfies
\begin{eqnarray}\label{Condition_1}
\|\lambda u-\lambda_h^{(\ell)}u_h^{(\ell)}\|_b&\leq&
\bar C_\lambda \eta_a(V_H)\|u - u_h^{(\ell)}\|_a.\label{Estimate_h_k_1_b}
\end{eqnarray}

Then the eigenpair approximation $(\lambda_h^{(\ell+1)},u_h^{(\ell+1)})\in\mathbb R\times V_h$ obtained
by Algorithm \ref{one correction step_Eig} satisfies
\begin{eqnarray}
\|u - u_h^{(\ell+1)}\|_a &\leq & \gamma\ \|u-u_h^{(\ell)}\|_a+\zeta\ \|u-\mathcal P_hu\|_a,\label{Estimate_h_k_1_a}\\
\|\lambda u-\lambda_h^{(\ell+1)}u_h^{(\ell+1)}\|_b&\leq&
\bar C_\lambda \eta_a(V_H)\|u - u_h^{(\ell+1)}\|_a,\label{Estimate_h_k_1_b}
\end{eqnarray}
where the constants $\gamma$, $\zeta$, $\bar C_\lambda$ and $\bar D_\lambda$ are defined as
\begin{eqnarray}
\gamma &=& \frac{1}{1-\bar D_\lambda \eta_a(V_H)}
\Big(\theta+(1+\theta)\frac{\bar C_\lambda}{\sqrt{\lambda_1}}\eta_a(V_H)\Big)\,,\label{Gamma_Definition}\\
\zeta &=&\frac{1+\theta}{1-\bar D_\lambda \eta_a(V_H)},\label{definition_Zeta}\\
\bar C_\lambda &=& 2|\lambda|\Big(1+\frac{1}{\lambda_1\delta_\lambda}\Big)
+ \bar\lambda_{i,H}\sqrt{1+\frac{1}{\lambda_1\delta_\lambda^2}\eta_a^2(V_H)},\label{Definition_C_Bar}\\
\bar D_\lambda &=& \frac{1}{\sqrt{\lambda_1}}\left(2|\lambda|\Big(1+\frac{1}{\lambda_1\delta_\lambda}\Big)
+ \bar\lambda_{i,H}\sqrt{1+\frac{1}{\lambda_1\delta_\lambda^2}\eta_a^2(V_H)}\right).\label{Definition_D_Lambda_Bar}
\end{eqnarray}
\end{theorem}
\begin{proof}
From \eqref{Smallest_Eigenvalue}, (\ref{weak_eigenvalue_problem}), (\ref{Definition_FEM_Projection}),
(\ref{correct_source_exact_para}) and (\ref{Condition_1}), the following estimate holds for any $w\in V_h$,
\begin{eqnarray}\label{One_Correction_1}
&&a(\mathcal P_hu-\widehat{u}_h^{(\ell+1)}, w)
=b\big((\lambda u-\lambda_h^{(\ell)}u_h^{(\ell)}),w\big)
\leq\|\lambda u-\lambda_h^{(\ell)}u_h^{(\ell)}\|_b\|w\|_b\nonumber\\
&&\leq \bar C_\lambda\eta_a(V_H)\| u-u_h^{(\ell)}\|_a \|w\|_b
\leq\frac{1}{\sqrt{\lambda_1}}\bar C_\lambda\eta_a(V_H)\|u-u_h^{(\ell)}\|_a \|w\|_a.
\end{eqnarray}
Taking $w = \mathcal P_hu-\widehat{u}_h^{(\ell+1)}$ in (\ref{One_Correction_1}) implies the following estimate
\begin{eqnarray}\label{One_Correction_2}
\|\mathcal P_hu-\widehat{u}_h^{(\ell+1)}\|_a &\leq&\frac{\bar C_\lambda}{\sqrt{\lambda_1}}\eta_a(V_H)\| u-u_h^{(\ell)}\|_a.
\end{eqnarray}
Combining \eqref{Contraction_Rate} with \eqref{One_Correction_2}, it follows that
\begin{eqnarray}\label{Error_3}
\|\mathcal P_hu-\widetilde{u}_h^{(\ell+1)}\|_a&\leq& \|\mathcal P_hu-\widehat{u}_h^{(\ell+1)}\|_a
+ \|\widetilde{u}_h^{(\ell+1)}-\widehat{u}_h^{(\ell+1)}\|_a\nonumber\\
&\leq& \|\mathcal P_hu-\widehat{u}_h^{(\ell+1)}\|_a  + \theta \|\widehat{u}_h^{(\ell+1)}-u_h^{(\ell)}\|_a\nonumber\\
&\leq& \|\mathcal P_hu-\widehat{u}_h^{(\ell+1)}\|_a +
\theta \|\widehat{u}_h^{(\ell+1)}- \mathcal P_hu\|_a
+\theta\|\mathcal P_hu -u_h^{(\ell)}\|_a\nonumber\\
&\leq& (1+\theta)\|\mathcal P_hu-\widehat{u}_h^{(\ell+1)}\|_a
+\theta\|\mathcal P_hu - u\|_a +\theta\|u-u_h^{(\ell)}\|_a\nonumber\\
&\leq& \Big(\theta+(1+\theta)\frac{\bar C_\lambda}{\sqrt{\lambda_1}}\eta_a(V_H)\Big)\|u-u_h^{(\ell)}\|_a
+\theta\|u-\mathcal P_hu\|_a.
\end{eqnarray}
Similarly,
the discrete eigenvalue problem \eqref{parallel_correct_eig_exact} can be regarded as a subspace approximation to
the eigenvalue problem (\ref{weak_eigenvalue_problem}). Thus, from (\ref{Err_Norm_1_Superclose}), (\ref{Error_3}), 
Lemma  \ref{Error_Estimate_Theorem} and Corollary \ref{Error_Superclose_Theorem}, there hold following error estimates
\begin{eqnarray*}\label{Error_u_u_h_2}
\|u-u_h^{(\ell+1)}\|_a &\leq& \frac{1}{1-\bar D_\lambda \eta_a(V_{H,h})}\inf_{v_{H,h_k}\in V_{H,h}}\|u-v_{H,h_k}\|_a\nonumber\\
&\leq& \frac{1}{1-\bar D_\lambda \eta_a(V_H)}\|u-\widetilde{u}_h^{(\ell+1)}\|_a \nonumber\\
&\leq&  \frac{1}{1-\bar D_\lambda \eta_a(V_H)} \big(\|u-\mathcal P_hu\|_a + \|\mathcal P_hu-\widetilde{u}_h^{(\ell+1)}\|_a \big)\nonumber\\
&\leq& \gamma\ \|u-u_h^{(\ell)}\|_a + \zeta\ \|u-\mathcal P_hu\|_a,
\end{eqnarray*}
and
\begin{eqnarray*}\label{Error_u_u_h_2_Negative}
\|\lambda u-\lambda_h^{(\ell+1)} u_h^{(\ell+1)}\|_b
\leq \bar C_\lambda\eta(V_{H,h_k})\|u-u_h^{(\ell+1)}\|_a
\leq \bar C_\lambda\eta(V_H)\|u_h-u_h^{(\ell+1)}\|_a.
\end{eqnarray*}
Thus the desired results (\ref{Estimate_h_k_1_a}) and (\ref{Estimate_h_k_1_b}) are obtained
and the proof is complete.
\end{proof}
Eventhough there is no nested sequence of meshes, some efficient numerical algorithms
such as algebraic multigrid (AMG) method can be adopted as the linear solver for (\ref{correct_source_exact_para}).
\begin{corollary}\label{Error_Estimate_Corollary}
Under the conditions of Theorem \ref{Error_Estimate_One_Smoothing_Theorem},
%Augmented subspace iteration algorithm
after executing $L$ augmented subspace iteration step defined by Algorithm \ref{one correction step_Eig},
the resultant eigenpair approximation $(\lambda_h^{(L)},u_h^{(L)})\in\mathbb R\times V_h$ has following error estimates
\begin{eqnarray}
\|u - u_h^{(L)}\|_a &\leq & \gamma^L\|u-u_h^{(0)}\|_a + \frac{1-\gamma^L}{1-\gamma} \zeta \|u-\mathcal P_hu\|_a,\label{Estimate_h_L}\\
\|\lambda u-\lambda_h^{(L)}u_h^{(L)}\|_b&\leq&
\bar C_\lambda \eta_a(V_H)\|u - u_h^{(L)}\|_a.\label{Estimate_h_L_L2}
\end{eqnarray}
\end{corollary}
\begin{proof}
According to (\ref{Estimate_h_k_1_a}) and recursive argument, it follows that
\begin{eqnarray*}
\|u-u_h^{(L)}\|_a &\leq& \gamma \|u-u_h^{(L-1)}\|_a  +\zeta \|u-\mathcal P_hu\|_a\nonumber\\
&\leq& \gamma\big( \gamma  \|u-u_h^{(L-2)}\|_a+\zeta \|u-\mathcal P_h\|_a\big)   + \zeta \|u-\mathcal P_hu\|_a\nonumber\\
&\leq& \gamma^L\|u-u_h^{(0)}\|_a + \sum_{\ell=0}^L\gamma^\ell \zeta \|u-\mathcal P_hu\|_a\nonumber\\
&=&\gamma^L\|u-u_h^{(0)}\|_a + \frac{1-\gamma^L}{1-\gamma} \zeta \|u-\mathcal P_hu\|_a,
\end{eqnarray*}
which proves the inequality (\ref{Estimate_h_L}).
Similarly to the proof of Theorem \ref{Error_Estimate_One_Smoothing_Theorem}, the desired result (\ref{Estimate_h_L_L2})
can also be deduced.
\end{proof}
From the convergence results of Theorem \ref{Error_Estimate_One_Smoothing_Theorem} and the definition (\ref{Gamma_Definition}),
it is easy to know that $\gamma$ is less than $1$ and independent of the finer mesh size $h$ when $H$ is sufficiently small.
\revise{\begin{remark}
The eigenpair solution $(\lambda_h^{(\ell+1)},u_h^{(\ell+1)})$ of (\ref{parallel_correct_eig_exact}) is an
algebraic approximation to the following eigenvalue problem:
Find $(\bar{\lambda}_{H,h}, \bar{u}_{H,h})\in \mathcal{R}\times (V_H+V_h\backslash V_H)$
such that $a(\bar{u}_{H,h},\bar{u}_{H,h})=1$ and
\begin{eqnarray}\label{Weak_Eigenvalue_Discrete_Hh}
a(\bar{u}_{H,h},v_{H,h})=\bar{\lambda}_{H,h} b(\bar{u}_{H,h},v_{H,h}),\quad\ \  \ \forall\ v_{H,h}\in V_{H,h},
\end{eqnarray}
where $V_h\backslash V_H$ denotes the space which are produced by deleting the components in $V_H$ from $V_h$.
\end{remark}}

In Algorithm \ref{one correction step_Eig}, the space $V_{H,h}$ is defined based on $V_H$ on the coarse mesh and $V_h$ on the finer mesh.
Since $V_H\not\subset V_h$, the augmented subspace method defined by Algorithm \ref{one correction step_Eig} can be combined
with the moving mesh method where the sequence of meshes does not have nested property \cite{DiLiTangZhang,LiTangZhang_1,LiTangZhang_2,Miller_2,Miller_1}.
This is the most important contribution of this paper.
Because of $V_H\not\subset V_h$, the definition of the interpolation operator $I_H^h: V_H\rightarrow V_h$ is
different from the standard one which is defined on the nested meshes $\mathcal T_H$ and $\mathcal T_h$.
For the detailed construction and implementation,
please refer to the documentation of finite element package FreeFem++ \cite{FreeFEM,FreeFEM_2}.

Now, we consider the details for solving the small scale eigenvalue problem (\ref{parallel_correct_eig_exact}).
Let $N_H$ and $\{\psi_{j,H}\}_{1\leq j\leq N_H}$ denote the dimension
and Lagrange basis functions for the coarse finite element space $V_H$.
The function in $V_{H,h}$ can be denoted by $u_{H,h}=u_H+\xi\widetilde u_h$.
Solving eigenvalue problem (\ref{parallel_correct_eig_exact}) is to obtain the
function $u_H\in V_H$ and the value $\xi\in \mathcal R$.
Let $u_H=\sum_{j=1}^{N_H}u_j\psi_{j,H}$ and define the vector $\mathbf u_H=[u_1,\cdots, u_{N_H}]^T$.
The corresponding matrix version of (\ref{parallel_correct_eig_exact}) can be defined as follows
\begin{eqnarray}\label{Eigenvalue_Problem_Hh}
\left(
\begin{array}{cc}
A_H & a_h\\
a_h^T & \alpha
\end{array}
\right)
\left(
\begin{array}{c}
\mathbf u_H \\
\xi
\end{array}
\right)
=\lambda_h
\left(
\begin{array}{cc}
B_H & b_h\\
b_h^T & \beta
\end{array}
\right)
\left(
\begin{array}{c}
\mathbf u_H \\
\xi
\end{array}
\right),
\end{eqnarray}
where $\mathbf u_H\in \mathcal R^{N_H\times 1}$ and $\xi\in\mathcal R$.

For understanding the proposed method, we introduce the assembling method for the matrices $A_H$ and $B_H$,
vectors $a_h$ and $b_h$, scalars $\alpha$ and $\beta$.

The matrix $A_H$ is defined as
\begin{eqnarray}\label{Definition_A_H}
(A_H)_{i,j} = (\mathcal A\nabla \psi_{i,H},\nabla\psi_{j,H})
= \int_{\Omega}\nabla \psi_{i,H}\cdot \mathcal A\nabla\psi_{j,H}d\Omega,\ \ \ 1\leq i,j\leq N_H.
\end{eqnarray}
In order to obtain the same precision as $V_h$, we need to
calculate the integral in (\ref{Definition_A_H}) on the finer mesh $\mathcal T_h$.
This is because we need to guarantee the accuracy for approximating the curved interfaces to reach the same
level as $\mathcal T_h$. Therefore, we use the following way
\begin{eqnarray}\label{Assembling_A_H}
(A_H)_{i,j} = \sum_{K\in\mathcal T_h}\int_K\nabla \psi_{i,H}\cdot\mathcal A\nabla\psi_{j,H}dK,\ \ \ 1\leq i,j\leq N_H.
\end{eqnarray}
Similarly, the assembling method for the mass matrix $B_H$ can be given as follows
\begin{eqnarray}\label{Assembling_B_H}
(B_H)_{i,j} = \sum_{K\in\mathcal T_h}\int_K\psi_{i,H}\psi_{j,H}dK,\ \ \ 1\leq i,j\leq N_H.
\end{eqnarray}

Now we concentrate on assembling the vector $a_h$, which is defined as follows
\begin{eqnarray}\label{Definition_a_h}
(a_h)_i = \int_{\Omega}\nabla\psi_{i,H}\cdot\mathcal A\nabla\widetilde u_hd\Omega, \ \ \ \ 1\leq i\leq N_H.
\end{eqnarray}
Since the finite element function $\widetilde u_h$ is defined on the finer mesh $\mathcal T_h$,
the assembling of  $a_h$ needs to be implemented on $\mathcal T_h$, i.e.,
\begin{eqnarray}\label{Assembling_a_h}
(a_h)_i = \sum_{K\in\mathcal T_h}\int_K\nabla \psi_{i,H}\cdot\mathcal A\nabla\widetilde u_hdK,\ \ \ 1\leq i\leq N_H.
\end{eqnarray}
Similarly, the vector $b_h$ should be assembled in the following way
\begin{eqnarray}\label{Assembling_b_h}
(b_h)_i = \sum_{K\in\mathcal T_h}\int_K \psi_{i,H} \widetilde u_hdK,\ \ \ 1\leq i\leq N_H.
\end{eqnarray}
Based on the structure of $V_{H,h}$, the scalars $\alpha$ and $\beta$ are assembled as follows
\begin{eqnarray}
&&\alpha = \int_\Omega\nabla\widetilde u_h\cdot\mathcal A\nabla\widetilde u_hd\Omega
=\sum_{K\in\mathcal T_h}\int_K\nabla\widetilde u_h\cdot\mathcal A\nabla \widetilde u_h dK,\label{Assembling_alpha}\\
&&\beta = \int_\Omega|\widetilde u_h|^2d\Omega=\sum_{K\in\mathcal T_h}\int_K|\widetilde u_h|^2dK.\label{Assembling_beta}
\end{eqnarray}

After assembling the matrices $A_H$ and $B_H$, vectors $a_h$ and $b_h$, scalars $\alpha$ and $\beta$,
some algebraic eigensolver are adopted to solve the eigenvalue problem (\ref{Eigenvalue_Problem_Hh})
to obtain $\mathbf u_H$ and $\xi$. For the next iteration, the function $u_H +\xi \widetilde u_h^{(\ell+1)}$ should be interpolated
into the finite element space $V_h$. With the help of the interpolator operator $I_H^h$, we can obtain $u_h^{(\ell+1)}$
by the following way
\begin{eqnarray}
u_h^{(\ell+1)} = I_H^hu_H +\xi \widetilde u_h^{(\ell+1)}.
\end{eqnarray}

According to the definition of Algorithm \ref{one correction step_Eig} and the detailed implementing process,
it is easy to state the estimate of computational work for the nonnested augmented subspace method.
For this aim, we denote the degree of freedom of the finite element space $V_h$ as $N_h$.
\begin{theorem}\label{Optimal_Work_Eig}
Assume solving the linear eigenvalue problem (\ref{Eigenvalue_Problem_Hh}) needs work $\mathcal{O}(M_H)$ ($M_H>N_H$), and the work for
solving (\ref{correct_source_exact_para}) is  $\mathcal{O}(N_h)$.
Then the computational work included in Algorithm \ref{one correction step_Eig} is
\begin{eqnarray}\label{Computation_Work_Estimate_Eff}
{\rm Work}=\mathcal{O}\left(N_h+ M_H\right).
\end{eqnarray}
\end{theorem}

\section{Multilevel correction method}\label{MLCM}
Similarly to the full multigrid method for the linear boundary value problems,  we
can use the nonnested augmented subspace method defined by Algorithm \ref{one correction step_Eig} to build
a type of multilevel correction method for the eigenvalue problem (\ref{weak_eigenvalue_problem}).
Different from the existed multilevel correction method in \cite{LinXie_2010,LinXie,Xie_JCP,Xie_IMA},
the sequence of meshes has no nested property since the existence of the curved interfaces.
The idea to build the multilevel correction method is to use
the eigenpair approximations on the coarse mesh as the initial values
on the finer mesh for augmented subspace algorithm.
The reason to call the proposed method as the multilevel correction method is the sequence of
concerned finite element spaces has no nested property.

In order to design the multilevel correction method, we first introduce the
sequence of finite element spaces. We generate a coarse mesh $\mathcal T_H$ with the mesh size $H$ and the
coarse linear finite element space $V_H$ is defined on the mesh $\mathcal T_H$.
Then a sequence of meshes $\mathcal T_{h_k}$ is generated by some type of mesh tool and  the mesh sizes $h_k$
satisfy  the following properties
\begin{eqnarray}\label{mesh_size_recur}
h_1 <H,\ \ \ h_k=\frac{1}{\beta}h_{k-1},\ \ \ \ k=2, \cdots, n.
\end{eqnarray}
Based on the sequence of meshes $\mathcal T_{h_k}$,  we can construct the corresponding
linear finite element spaces $V_{h_k}$ ($k=1, \cdots, n$).
Although the sequence of spaces $V_{h_k}$ does not have nested properties,
the following relationships and error estimates hold
\begin{eqnarray}\label{delta_recur_relation}
\eta_a(V_{h_k})\approx \frac{1}{\beta}\eta_a(V_{h_{k-1}}),\ \ \ \
\delta(u, V_{h_{k-1}})\approx \frac{1}{\beta}\delta(u, V_{h_k}),\ \ \ k=2,\cdots,n.
\end{eqnarray}
The corresponding multilevel correction method is defined by Algorithm \ref{Multilevel_Correction_Eig}.
\begin{algorithm}[htb]
\caption{Multilevel correction method}\label{Multilevel_Correction_Eig}
\begin{enumerate}
\item Solve the eigenvalue problem on $V_{h_1}$: Find $(\lambda_{h_1}, u_{h_1})\in \mathbb R\times V_{h_1}$ such that
\begin{equation*}
a(u_{h_1}, v_{h_1})=\lambda_{h_1}(u_{h_1},v_{h_1}), \quad \forall\ v_{h_1}\in  V_{h_1}.
\end{equation*}
\item For $k=2,\cdots,n$, do the following iteration:
\begin{enumerate}
\item Let $u_{h_k}^{(0)}=u_{h_{k-1}}$ and $\lambda_{h_k}^{(0)}=\lambda_{h_{k-1}}$.
\item For $\ell=0, \cdots, L-1$, do the following augmented subspace iteration steps
\begin{eqnarray*}
(\lambda_{h_k}^{(\ell+1)}, u_{h_k}^{(\ell+1)})={\tt AugSubspace}(V_H,\lambda_{h_k}^{(\ell)}, u_{h_k}^{(\ell)},V_{h_k}).
\end{eqnarray*}
\item Define $u_{h_k}=u_{h_k}^{(L)}$ and $\lambda_{h_k}=u_{h_k}^{(L)}$.
\end{enumerate}
\end{enumerate}
\end{algorithm}

Based on Theorem \ref{Error_Estimate_One_Smoothing_Theorem},
Corollary \ref{Error_Estimate_Corollary} and the property (\ref{delta_recur_relation}), we can deduce
the error estimates for Algorithm \ref{Multilevel_Correction_Eig} with some recursive argument.
\begin{theorem}\label{Error_Multi_Correction_Theorem}
Under the condition of (\ref{delta_recur_relation}),
the eigenpair approximation $(\lambda_{h_n}, u_{h_n})\in\mathbb R\times V_{h_n}$
obtained by Algorithm \ref{Multilevel_Correction_Eig} has the following error estimate
\begin{eqnarray}\label{Error_Multi_Correction_a}
\|u-u_{h_n}\|_a &\leq \frac{1-(\beta\gamma^L)^n}{1-\beta\gamma^L}\mu\delta(u, V_{h_n}),
\end{eqnarray}
where
\begin{eqnarray}\label{Definition_mu}
\mu := \frac{1-\gamma^L}{1-\gamma} \zeta.
\end{eqnarray}
\end{theorem}
\begin{proof}
From Lemma \ref{Error_Estimate_Theorem} and the property $\eta_a(V_{h_1})\leq \eta_a(V_H)$,
the eigenfunction approximation $u_{h_1}$ obtained by Step 1 of Algorithm \ref{Multilevel_Correction_Eig}
satisfies the following error estimate
\begin{eqnarray}\label{Inequality_1}
\|u-u_{h_1}\|_a &\leq& \frac{1}{1-\bar D_\lambda\eta_a(V_{h_1})}\delta(u,V_{h_1}) \leq \mu \|u-\mathcal P_{h_1}u\|_a.
\end{eqnarray}
Combining Corollary \ref{Error_Estimate_Corollary}, (\ref{Inequality_1}) and recursive argument leads to the following estimates
\begin{eqnarray*}
\|u-u_{h_n}\|_a &\leq& \gamma^L\|u-u_{h_{n-1}}\|_a + \mu \|u-\mathcal P_{h_n}u\|_a \nonumber\\
&\leq&  \gamma^L\big(\gamma^L\|u-u_{h_{n-2}}\|_a + \mu \|u-\mathcal P_{h_{n-1}}u\|_a\big) +\mu \|u-\mathcal P_{h_n}u\|_a\nonumber\\
&\leq& \gamma^{(n-1)L}\|u-u_{h_1}\|_a + \sum_{k=2}^{n}\gamma^{(n-k)L}\mu\|u-\mathcal P_{h_k}u\|_a\nonumber\\
&\leq& \gamma^{(n-1)L}\mu\|u-\mathcal P_{h_1}u\|_a + \sum_{k=2}^{n}\gamma^{(n-k)L}\mu\|u-\mathcal P_{h_k}u\|_a\nonumber\\
&=& \sum_{k=1}^{n}\gamma^{(n-k)L}\mu\delta(u, V_{h_k})
\leq \left(\sum_{k=1}^{n}\gamma^{(n-k)L}\beta^{n-k}\right)\mu\delta(u, V_{h_n})\nonumber\\
&=&\left(\sum_{k=1}^{n}\big(\beta\gamma^{L}\big)^{n-k}\right)\mu\delta(u, V_{h_n})
\leq \frac{1-\big(\beta\gamma^L\big)^n}{1-\beta\gamma^L}\mu\delta(u, V_{h_n}).
\end{eqnarray*}
This is the desired result (\ref{Error_Multi_Correction_a}) and we complete the proof.
\end{proof}

Now we turn our attention to the estimate of computational work for Algorithm \ref{Multilevel_Correction_Eig}.
First, we define the dimension of each level of finite element space as $N_{h_k}:={\rm dim}V_{h_k}$.
Then the following property holds
\begin{eqnarray}\label{DOF_Relation}
N_{h_k}\approx\Big(\frac{1}{\beta}\Big)^{d(n-k)}N_{h_n},\ \ \ k=1, 2,\cdots, n.
\end{eqnarray}

\begin{theorem}\label{Work_Estimate_Multi_Correction}
Assume the conditions of Theorem \ref{Optimal_Work_Eig} hold
and solving the eigenvalue problem in $V_{h_1}$ needs work $\mathcal{O}(M_{h_1})$.
Then the computational work involved in Algorithm \ref{Multilevel_Correction_Eig} is
\begin{eqnarray}\label{Computation_Work_Estimate_Multi_Correction}
{\rm Total\ Work}=\mathcal{O}\big(LN_{h_n}+ M_{h_1}+LM_H\ln N_{h_n}\big).
\end{eqnarray}
\end{theorem}
\begin{proof}
From Theorem \ref{Optimal_Work_Eig} and (\ref{DOF_Relation}), it follows that
\begin{eqnarray*}
{\rm Total\ Work} &=&\mathcal O\left( M_{h_1} + \sum_{k=2}^n \big(L(N_{h_k}+M_H)\big)\right)\nonumber\\
&=&\mathcal O\left( M_{h_1} + L\sum_{k=2}^n\left(\Big(\frac{1}{\beta}\Big)^{n-k}N_{h_n}+M_H\right)\right)\nonumber\\
&=& \mathcal O\big( LN_{h_n}+M_{h_1}+LM_H\ln N_{h_n}\big).
\end{eqnarray*}
Thus the proof is complete.
\end{proof}

Based on the definition and the corresponding convergence theory,
we can find an interesting property that Algorithm \ref{one correction step_Eig} can work for only one single eigenpair.
During the multilevel correction process, there is no orthogonalization in the high dimension space $V_{h_k}$ with $k\geq 2$.
Because of avoiding doing the time-consuming orthogonalization in the high dimensional spaces,
the augmented subspace iteration algorithm improves the scalability for solving the eigenvalue problem.
Compared with the traditional eigensolvers based on the Krylov subspaces, the coarse space $V_H$ from
the augmented subspace $V_{H,h}$ has the approximation property to general functions (check the definition of $\eta_a(V_H)$ in (\ref{Definition_Eta_a_h})). 
This is obviously different from the property of the Krylov subspaces
which can only approximate the specific functions \cite{Saad}.
This is the reason why Algorithm \ref{one correction step_Eig} can compute one particular eigenpair approximation \cite{XuXieZhang}.

\section{Numerical examples}\label{NE}

In this section, we provide four numerical examples to validate the proposed augmented subspace algorithm and the corresponding
theoretical analysis. %for the eigenvalue problem with multiple curved interfaces.
\revise{Since the software FreeFEM++ offers a fast interpolation algorithm and a language
to manipulate the data on multiple meshes, the methods in this paper is implemented with FreeFEM++ \cite{FreeFEM,FreeFEM_2}}.
With the help of finite element package FreeFem++\cite{FreeFEM,FreeFEM_2}, the numerical experiments are carried out on LSSC-IV in the State Key
Laboratory of Scientific and Engineering Computing, Academy of Mathematics and Systems Science, Chinese Academy of Sciences.
Each computing node has two $18$-core Intel Xeon Gold $6140$ processors at $2.3$ GHz and $192$ GB memory.
The linear equation (\ref{correct_source_exact_para}) in Algorithm \ref{one correction step_Eig}
is solved by the package PETSc \cite{petsc-web-page,petsc-user-ref,petsc-efficient} with the aggregation-based AMG from Hypre (BoomerAMG) \cite{hypre}.
Each AMG step includes $5$ V-cycle with Falgout coarsening scheme,
one hybrid smoother from \revise{Symmetric Gauss Seidel} and Jacobi iterations.
The eigenvalue problem (\ref{parallel_correct_eig_exact}) is solved by the Krylov-Schur algorithm from Slepc \cite{SLEPc}.
\revise{Here, the eigenpair approximation $(\bar\lambda_h,\bar u_h)$ of (\ref{Weak_Eigenvalue_Discrete}) is chosen as
the exact one eigenpair to measure the errors of the approximations by the proposed algorithms.}

\subsection{Two dimensional examples}
In the first subsection, we investigate the convergence and efficiency of Algorithms \ref{one correction step_Eig}
and \ref{Multilevel_Correction_Eig} for two dimensional eigenvalue problems.

\subsubsection*{Example 1}
In the first example, we consider the elliptic eigenvalue problem with a piecewise constant coefficient and the computing
domain has two circle interfaces.
The nonnested augmented subspace method defined by Algorithm \ref{one correction step_Eig} is adopted to solve the following
eigenvalue problem: Find $(\lambda, u)$ such that
\begin{eqnarray}\label{Eigenvalue_Problem_2d}
\left\{
\begin{array}{rcl}
-\nabla\cdot(\mathcal K\nabla u) &=&\lambda u,\ \ \  \textrm{in}\ \Omega,\\
{[}u{]}=0, \ \ {\big[}\mathbf n_\Gamma \mathcal A\nabla u{\big]}&=&0, \ \ \ \quad \text{on}\  \Gamma,\\
u&=&0,\ \ \ \ \ \textrm{on}\ \partial\Omega.
\end{array}
\right.
\end{eqnarray}
Here, the computing domain $\Omega = (0,2)\times (0,2)$ includes two circles $\Omega_1$ and $\Omega_2$ with radius size $0.5$
and centers $(2/3,1)$ and $(4/3,1)$, respectively. The coefficient $\mathcal K$ in (\ref{Eigenvalue_Problem_2d}) is defined as
follows
\begin{eqnarray}
\mathcal K = \left\{
\begin{array}{ll}
10, & \textrm{in}\  \Omega_1=\{(x,y)\in\mathbb R^2| (x-2/3)^2+(y-1
)^2\leq 1/9\},\\
10, & \textrm{in}\  \Omega_2=\{(x,y)\in\mathbb R^2| (x-4/3)^2+(y-1)^2\leq 1/9\},\\
1,  & \textrm{in}\  \Omega_3 = \Omega/(\bar\Omega_1\cup\bar\Omega_2).
\end{array}
\right.
\end{eqnarray}

In order to check the effect of the coarse mesh $\mathcal T_H$ on the convergence rate, which is shown
in Theorems \ref{Error_Estimate_One_Smoothing_Theorem} and \ref{Error_Multi_Correction_Theorem},
Corollary \ref{Error_Estimate_Corollary}, we select two coarse meshes shown in Figure \ref{Exam_1_Coarse_Meshes} for the test.
For comparison, the finest mesh is chosen with the same $364416$ elements for the two cases of coarse meshes.
\begin{figure}[!hbt]
\centering
\includegraphics[width=7cm,height=5cm]{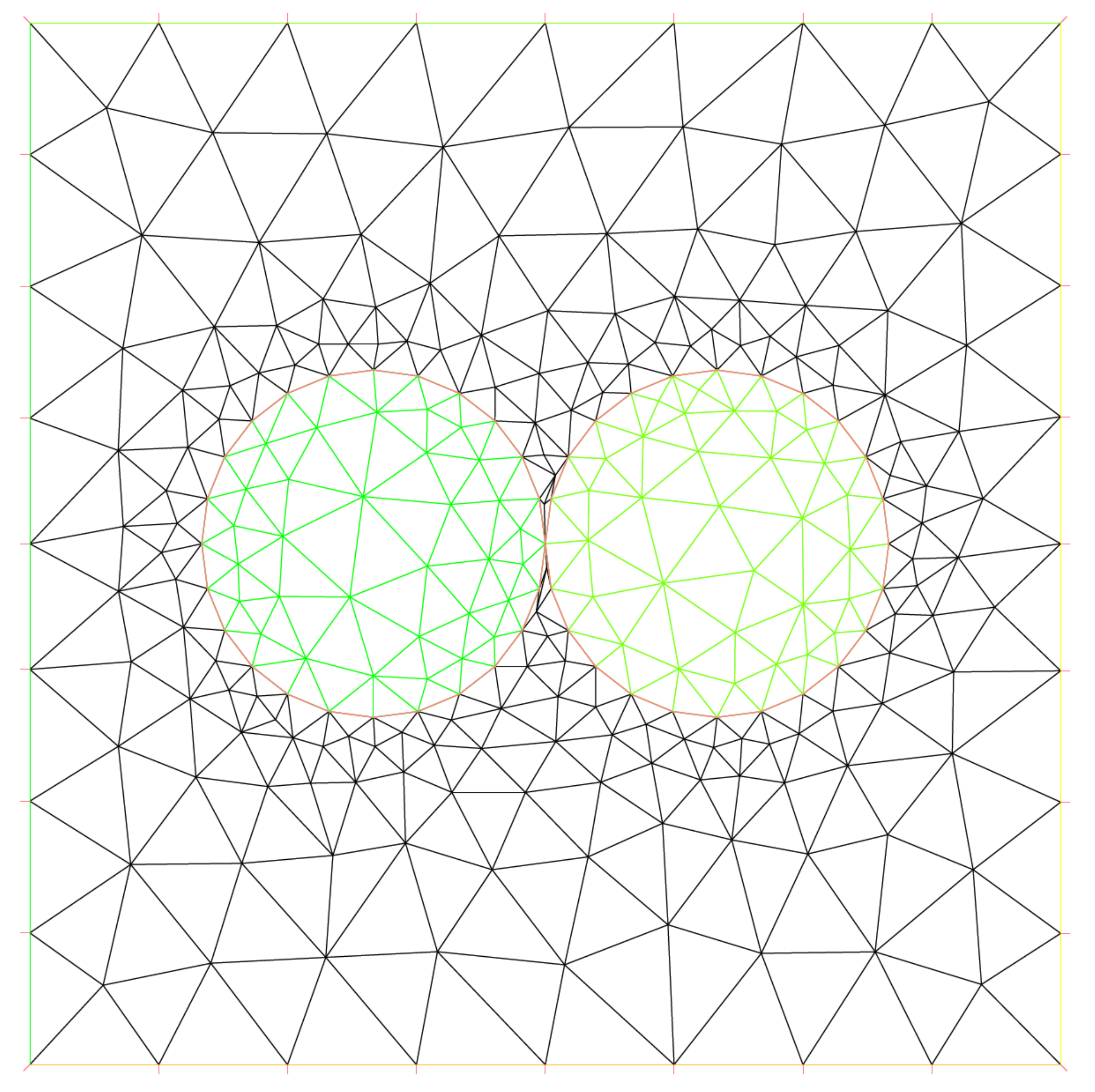}
\hskip-0cm
\includegraphics[width=7cm,height=5cm]{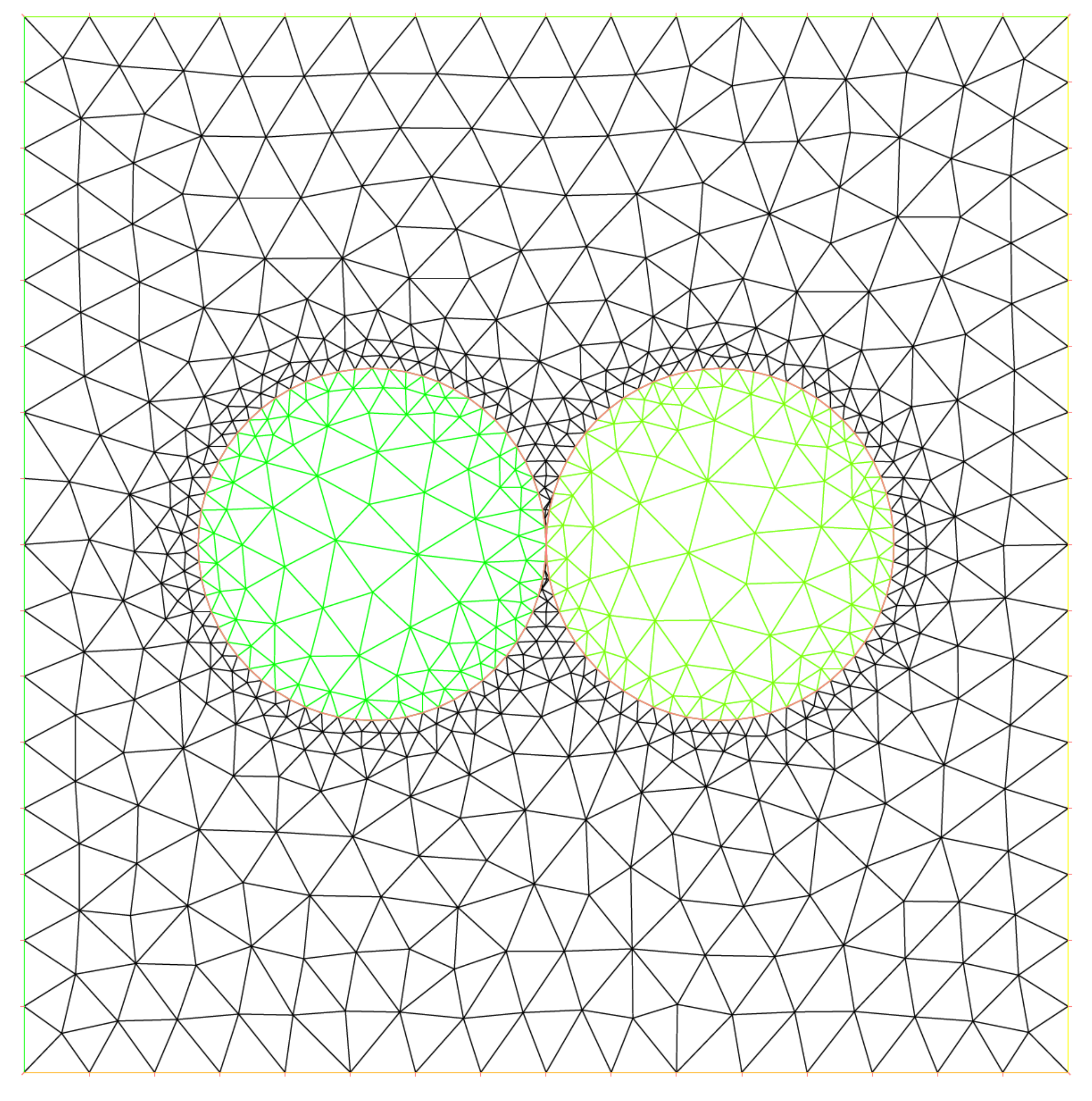}
\caption{Two coarse meshes $\mathcal T_H$ for Example 1: The left coarse mesh consists of $550$ elements, and the right
one $1456$ elements.}\label{Exam_1_Coarse_Meshes}
\end{figure}

Here, we check the numerical results for the first $4$ eigenfunctions and $10$ eigenvalues.
Since the second and third exact eigenvalues are multiple, we need to do the following spectral projection for the eigenfunction approximations $u_{2,h}$ and $u_{3,h}$ as follows:
\begin{eqnarray*}
a(E_{2,3}w,v_h) = a(w,v_h),\ \ \ \ \forall v_h\in {\rm span}\{\bar u_{1,h}, \bar u_{2,h}\}.
\end{eqnarray*}
Then the error estimate for  the first $4$ eigenfunction approximations can be defined as
\begin{eqnarray*}
\|u_{1,h}-\bar u_{1,h}\|_a +\|u_{2,h}-E_{2,3}u_{2,h}\|_a+\|u_{3,h}-E_{2,3}u_{3,h}\|_a+\|u_{4,h}-\bar u_{4,h}\|_a,
\end{eqnarray*}
where $\bar u_{i,h}$ ($1\leq i\leq 4$) denote the first $4$ exact finite element eigenfunctions defined
on the corresponding finer mesh $\mathcal T_h$.

When the coarse mesh is chosen as the left on in Figure \ref{Exam_1_Coarse_Meshes},
the corresponding numerical results are shown in Figure \ref{Exam_1_Error_H8}.
\begin{figure}[!hbt]
\centering
\includegraphics[width=7cm,height=5cm]{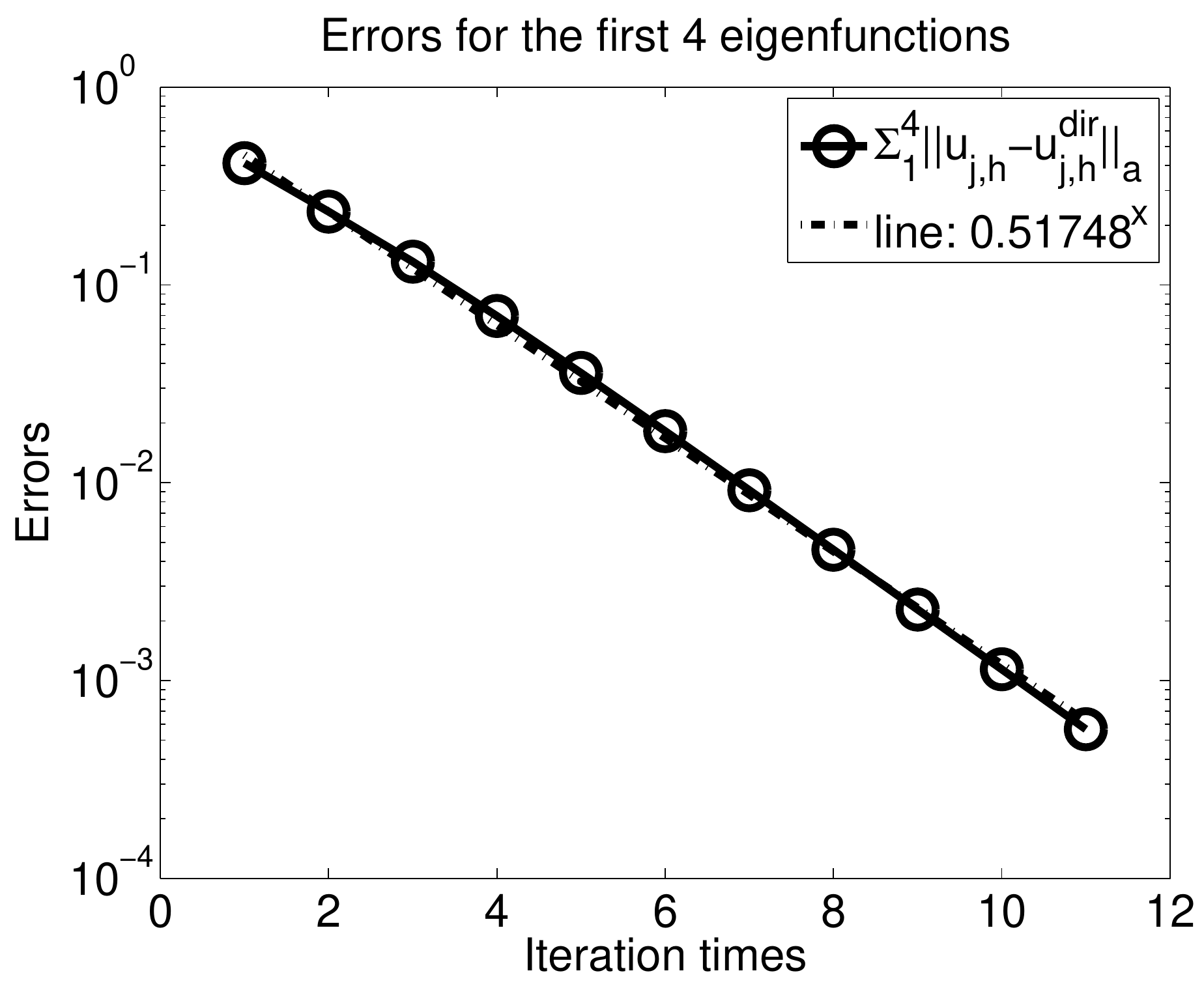}
\includegraphics[width=7cm,height=5cm]{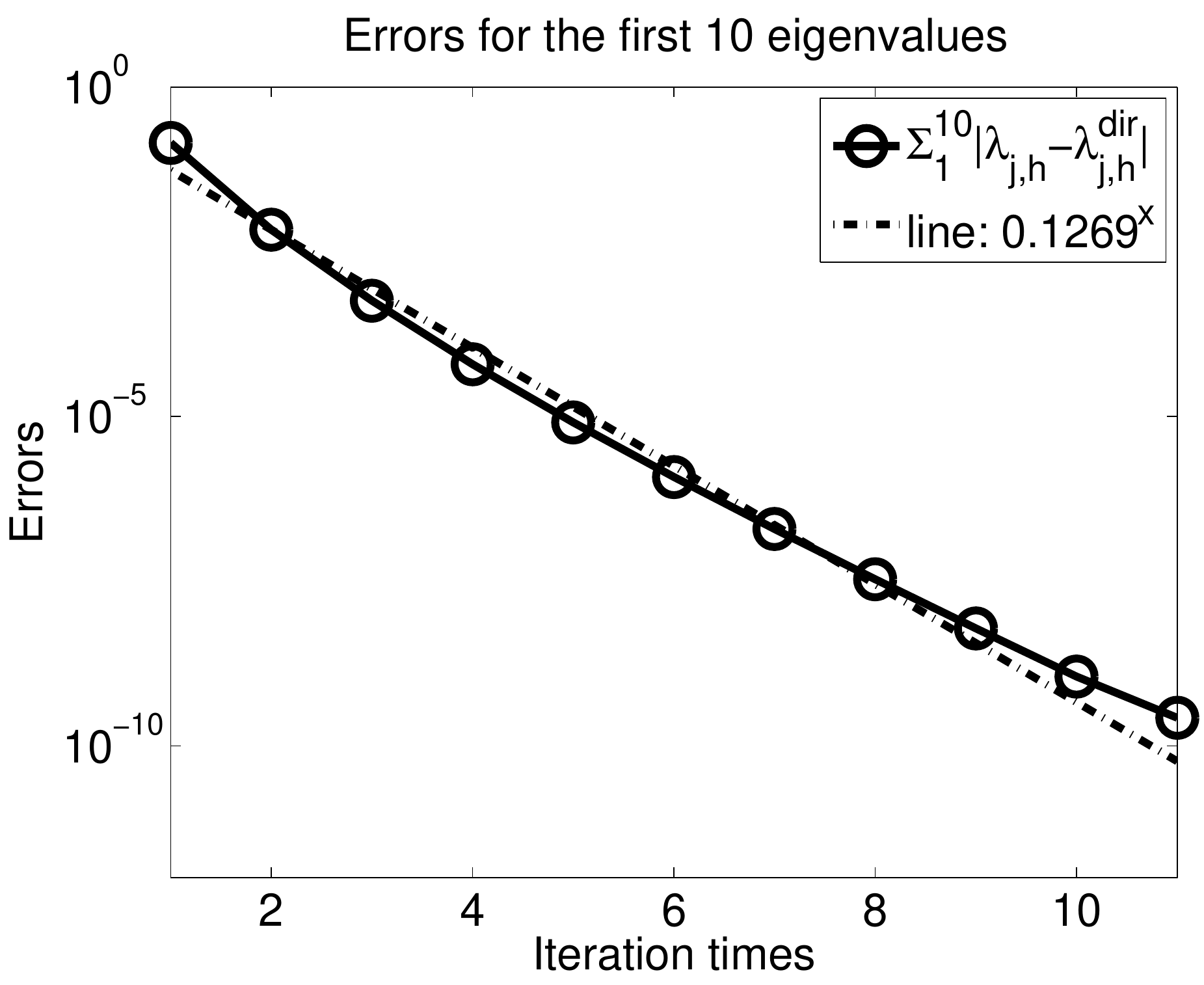}
\caption{Error estimates for the first $4$ eigenfunction and $10$ eigenvalue approximations by
Algorithm \ref{Multilevel_Correction_Eig}. Here the coarse mesh is chosen as the left one in Figure \ref{Exam_1_Coarse_Meshes}.}\label{Exam_1_Error_H8}
\end{figure}
Figure \ref{Exam_1_Error_H16} presents the corresponding numerical results for the coarse mesh is chosen as the right one in
Figure \ref{Exam_1_Coarse_Meshes}.
\begin{figure}[!hbt]
\centering
\includegraphics[width=7cm,height=5cm]{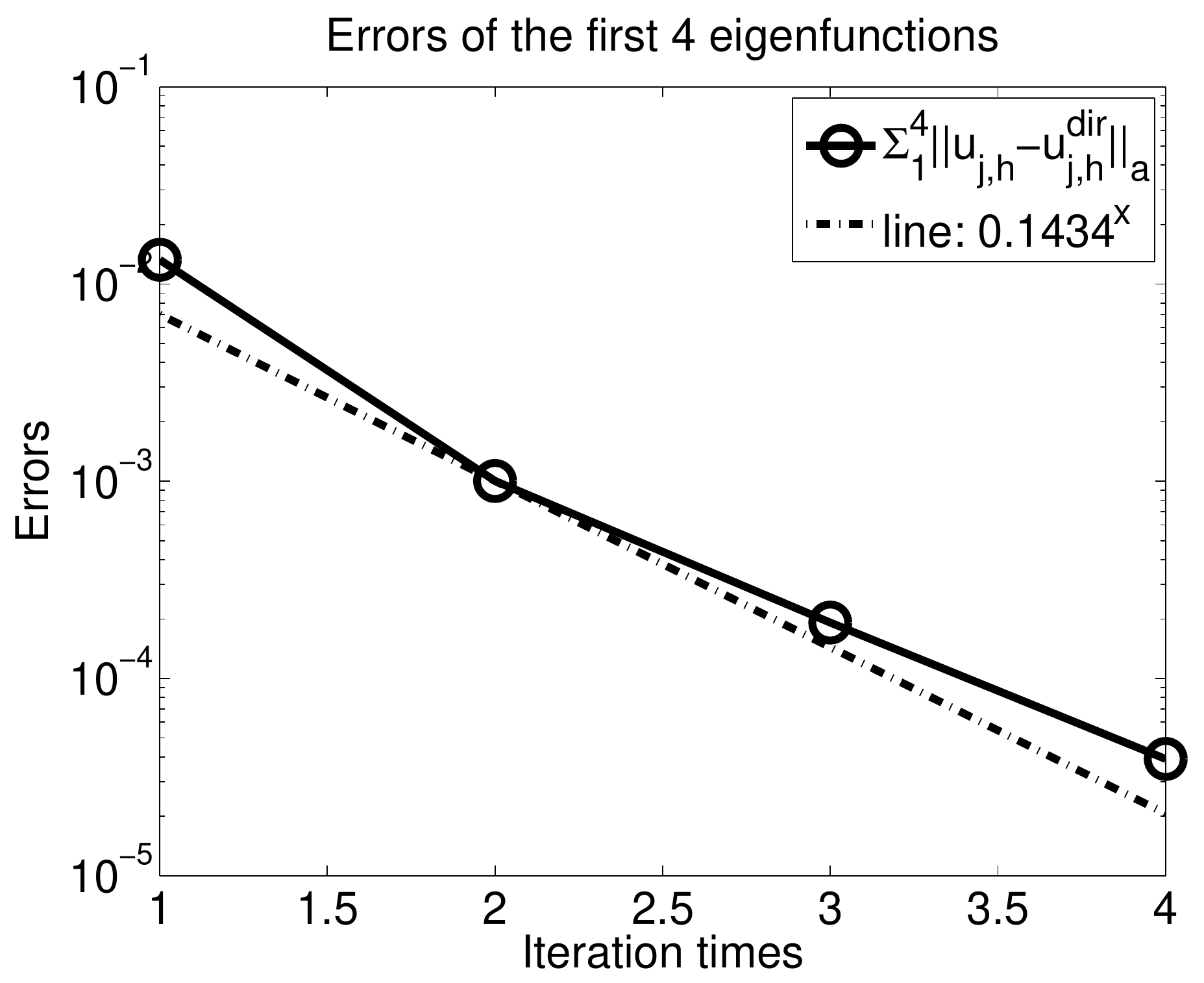}
\includegraphics[width=7cm,height=5cm]{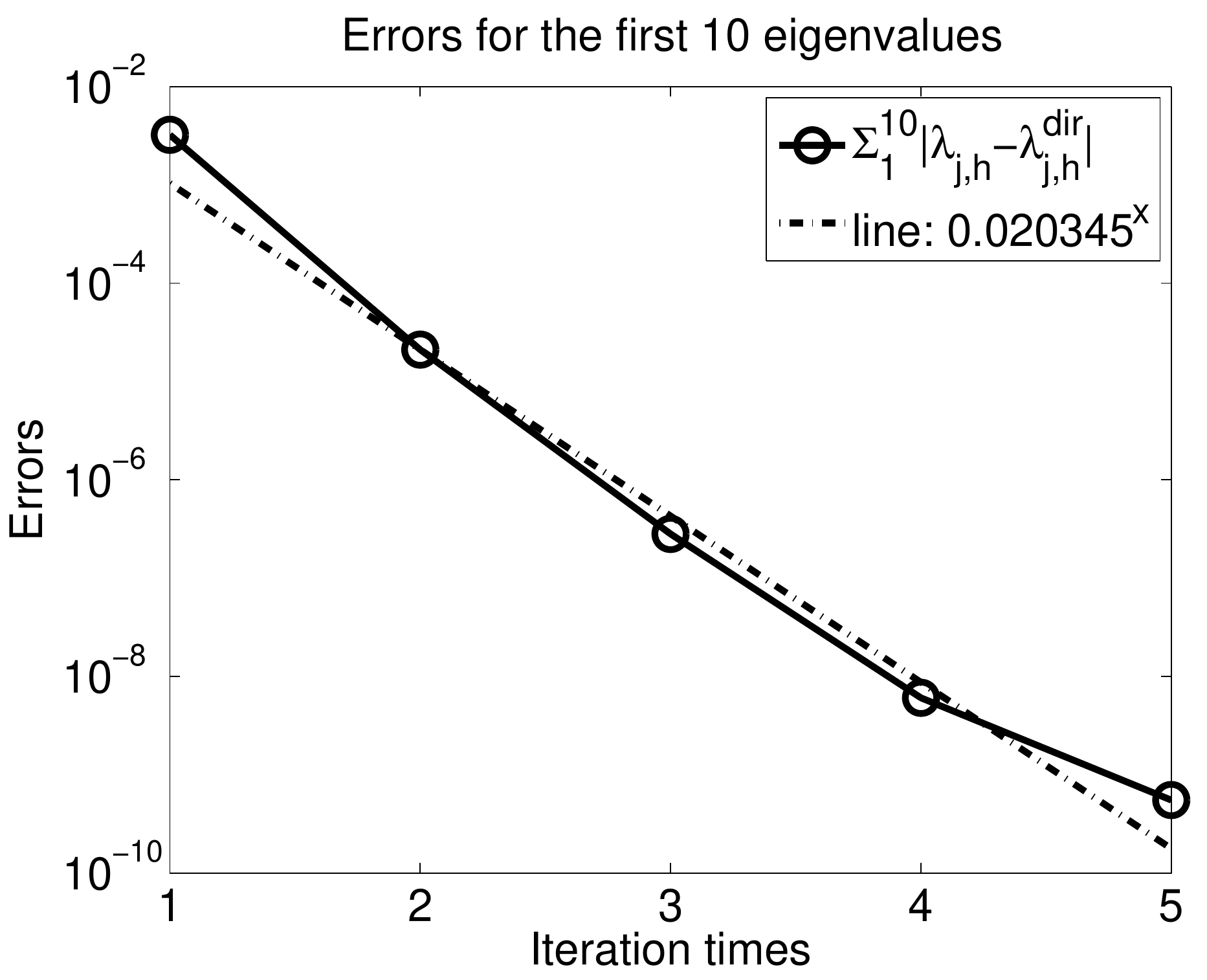}
\caption{Error estimates for the first $4$ eigenfunction and $10$ eigenvalue approximations by
Algorithm \ref{Multilevel_Correction_Eig}. Here the coarse mesh is chosen as the right one in Figure \ref{Exam_1_Coarse_Meshes}.}
\label{Exam_1_Error_H16}
\end{figure}
From Figures \ref{Exam_1_Error_H8} and \ref{Exam_1_Error_H16},  we can find that the finer mesh $\mathcal T_H$ has
faster convergence speed which validates theoretical results in Theorems \ref{Error_Estimate_One_Smoothing_Theorem}
and \ref{Error_Multi_Correction_Theorem}, Corollary \ref{Error_Estimate_Corollary}.

Furthermore, in order to check the efficiency of the proposed algorithms, we also investigate the CPU time for computing
the first $10$ eigenpair approximations.
Here, the convergence criterion is set to be $|\lambda_h-\bar \lambda_h|< 1$e-$9$.
Figure \ref{Exam_1_CPUTime} shows the corresponding CPU time when the coarse meshes are chosen as the two
in Figure \ref{Exam_1_Coarse_Meshes}.
The results in Figure \ref{Exam_1_CPUTime} validate the estimate of computational work in Theorem \ref{Work_Estimate_Multi_Correction}.
\begin{figure}[!hbt]
\centering
\includegraphics[width=7cm,height=5cm]{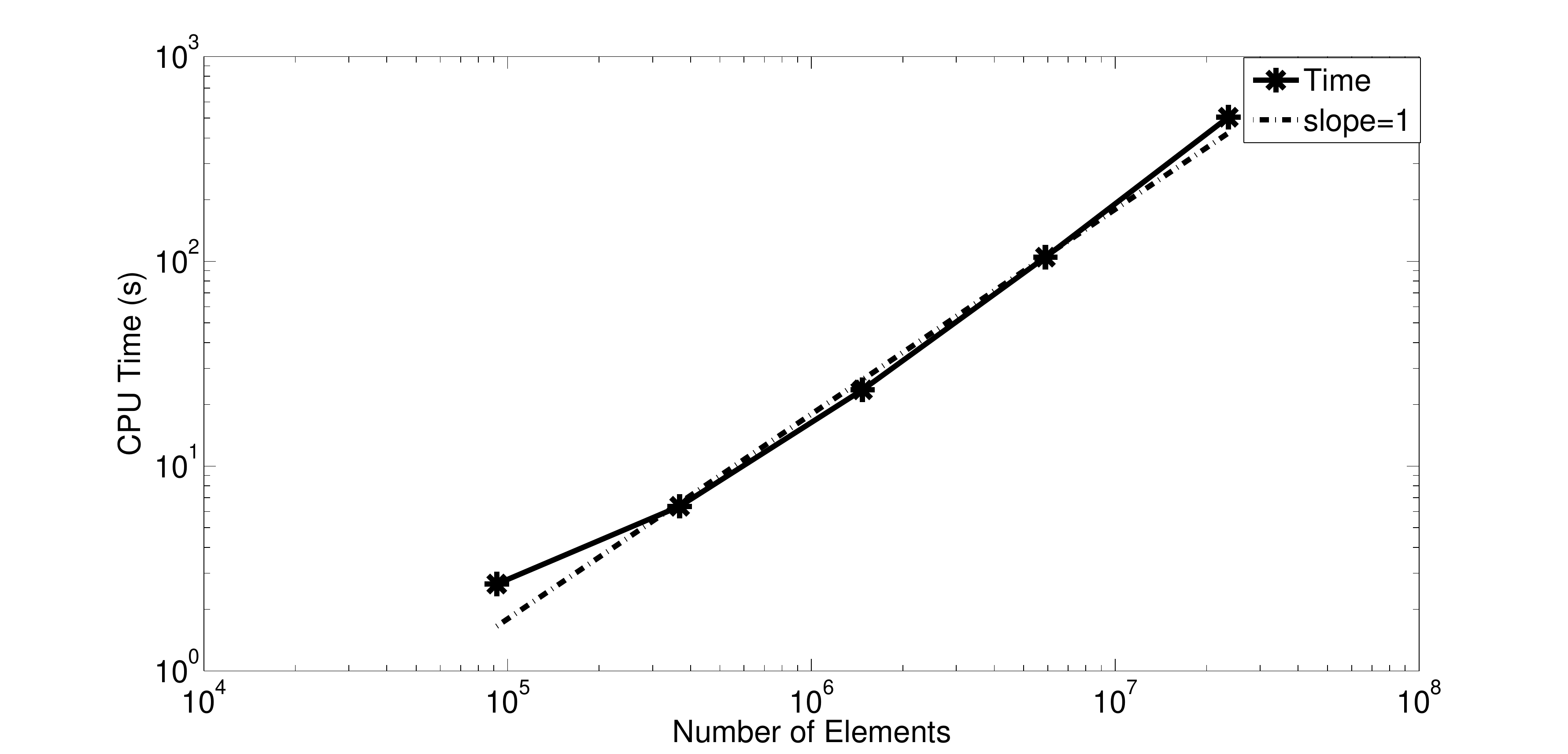}
\includegraphics[width=7cm,height=5cm]{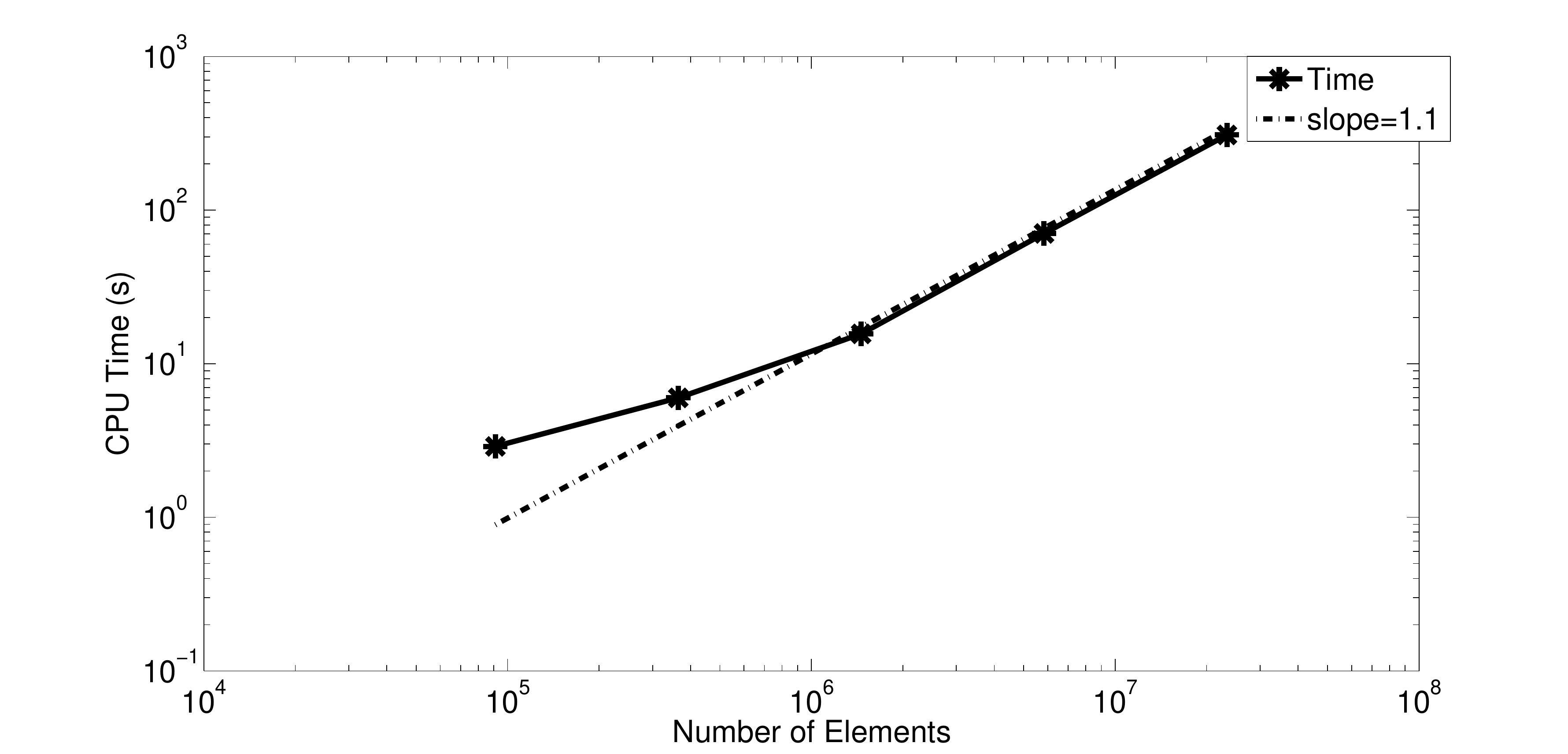}
\caption{CPU time for Algorithm \ref{Multilevel_Correction_Eig} \revise{with $32$ processors}, the left subfigure shows the CPU time 
when the coarse mesh is chosen as the left one in  
Figure \ref{Exam_1_Coarse_Meshes} and the right subfigure shows the CPU time 
when the coarse mesh is chosen as the right one in Figure \ref{Exam_1_Coarse_Meshes}.}\label{Exam_1_CPUTime}
\end{figure}

\subsubsection*{Example 2}
In the second example, we also solve the  eigenvalue problem (\ref{Eigenvalue_Problem_2d}).
Here, the computing domain $\Omega = (0,2)\times (0,2)$ is partitioned into
five parts by four circles with the radius $0.25$ and centers $(0.5, 0.5)$, $(1.5, 0.5)$, $(0.5, 1.5)$ and $(1.5,1.5)$, respectively.
The coefficient $\mathcal K$ in (\ref{Eigenvalue_Problem_2d}) is defined as follows
\begin{eqnarray*}
\mathcal K = \left\{
\begin{array}{ll}
10, & \textrm{in}\ \Omega_1=\{(x,y)\in\mathbb R^2| (x-0.5)^2+(y-0.5)^2\leq 1/16\},\\
10, & \textrm{in}\ \Omega_2=\{(x,y)\in\mathbb R^2| (x-1.5)^2+(y-0.5)^2\leq 1/16\},\\
10, & \textrm{in}\ \Omega_3=\{(x,y)\in\mathbb R^2| (x-0.5)^2+(y-1.5)^2\leq 1/16\},\\
10, & \textrm{in}\ \Omega_4=\{(x,y)\in\mathbb R^2| (x-1.5)^2+(y-1.5)^2\leq 1/16\},\\
1,  & \textrm{in}\ \Omega_5=\Omega/(\bar\Omega_1\cup\bar\Omega_2\cup\bar\Omega_3\cup\bar\Omega_4).
\end{array}
\right.
\end{eqnarray*}

In order to investigate the effect of the coarse mesh $\mathcal T_H$ on the convergence rate of the
nonnested augmented subspace method, we also choose two meshes shown in Figure \ref{Exam_2_Coarse_Meshes} for the test.
For the comparison, we select the same finest mesh which consists of $441600$ elements for this example.
In this example, we check the convergence behavior for computing the first $12$ eigenfunction and $20$ eigenvalue approximations.
\begin{figure}[!hbt]
\centering
\includegraphics[width=7cm,height=5cm]{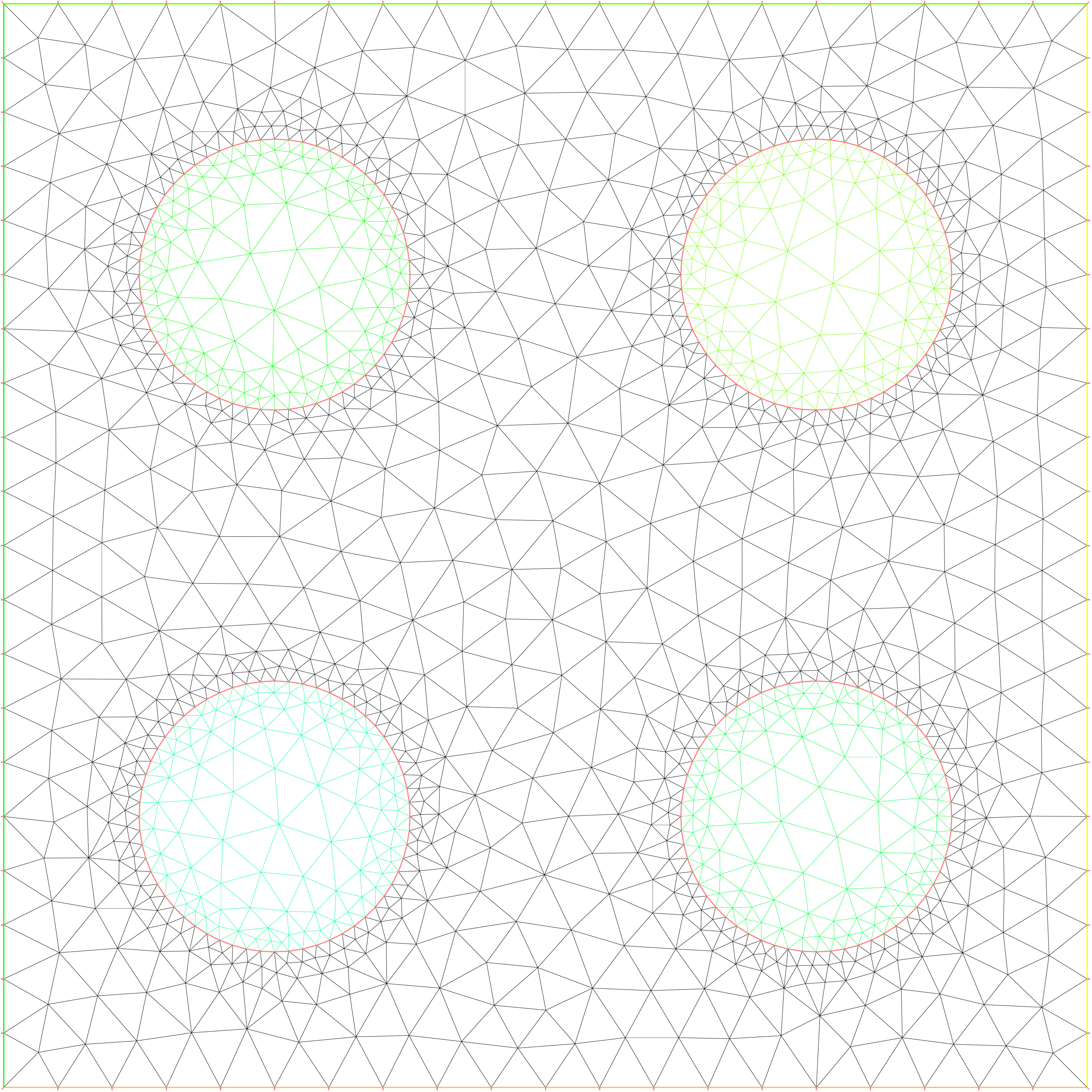}
\hskip-0cm
\includegraphics[width=7cm,height=5cm]{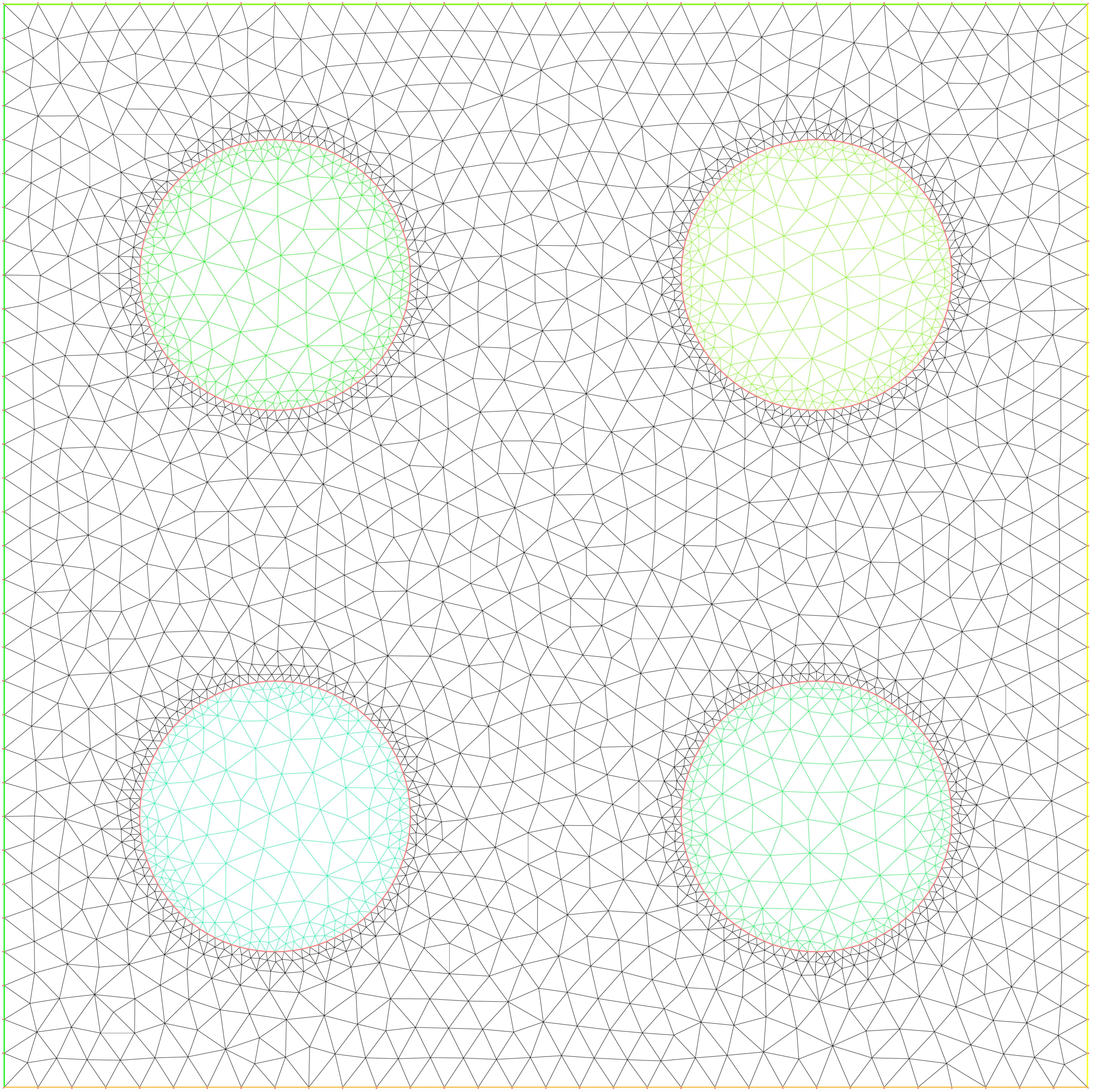}
\caption{Two coarse meshes $\mathcal T_H$ for Example 2: The left coarse mesh consists of $838$ elements, and the right
one $1992$ elements.}\label{Exam_2_Coarse_Meshes}
\end{figure}

When the left mesh of Figure \ref{Exam_2_Coarse_Meshes} acts as the coarse mesh $\mathcal T_H$,
Figure \ref{Exam_2_Error_H20} presents the corresponding numerical results for the first $12$ eigenfunction and
$20$ eigenvalue approximations. When the coarse mesh $T_H$ is chosen as the right one in Figure \ref{Exam_2_Coarse_Meshes},
the numerical results are shown in Figure \ref{Exam_2_Error_H32}.
\begin{figure}[!hbt]
\centering
\includegraphics[width=7cm,height=5cm]{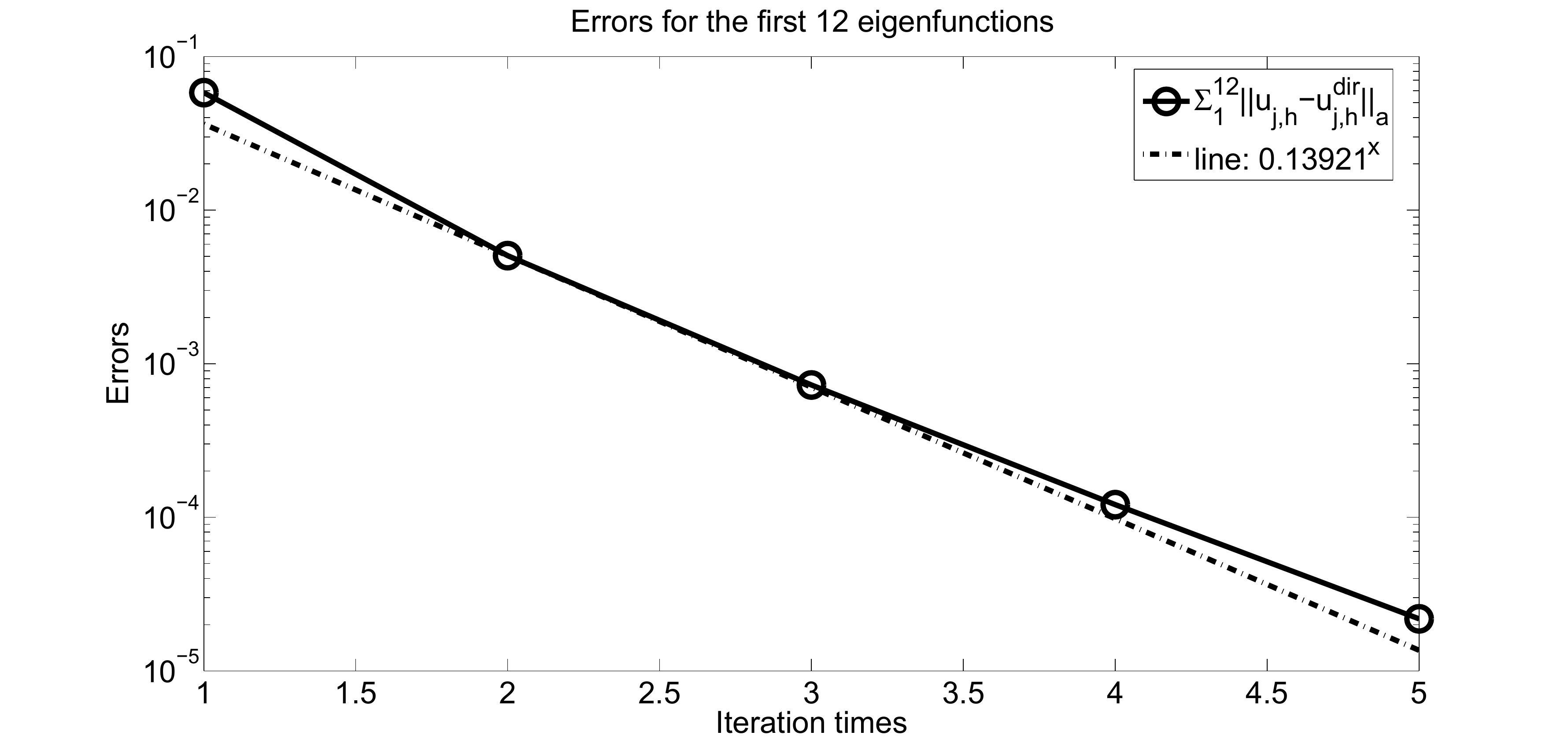}
\includegraphics[width=7cm,height=5cm]{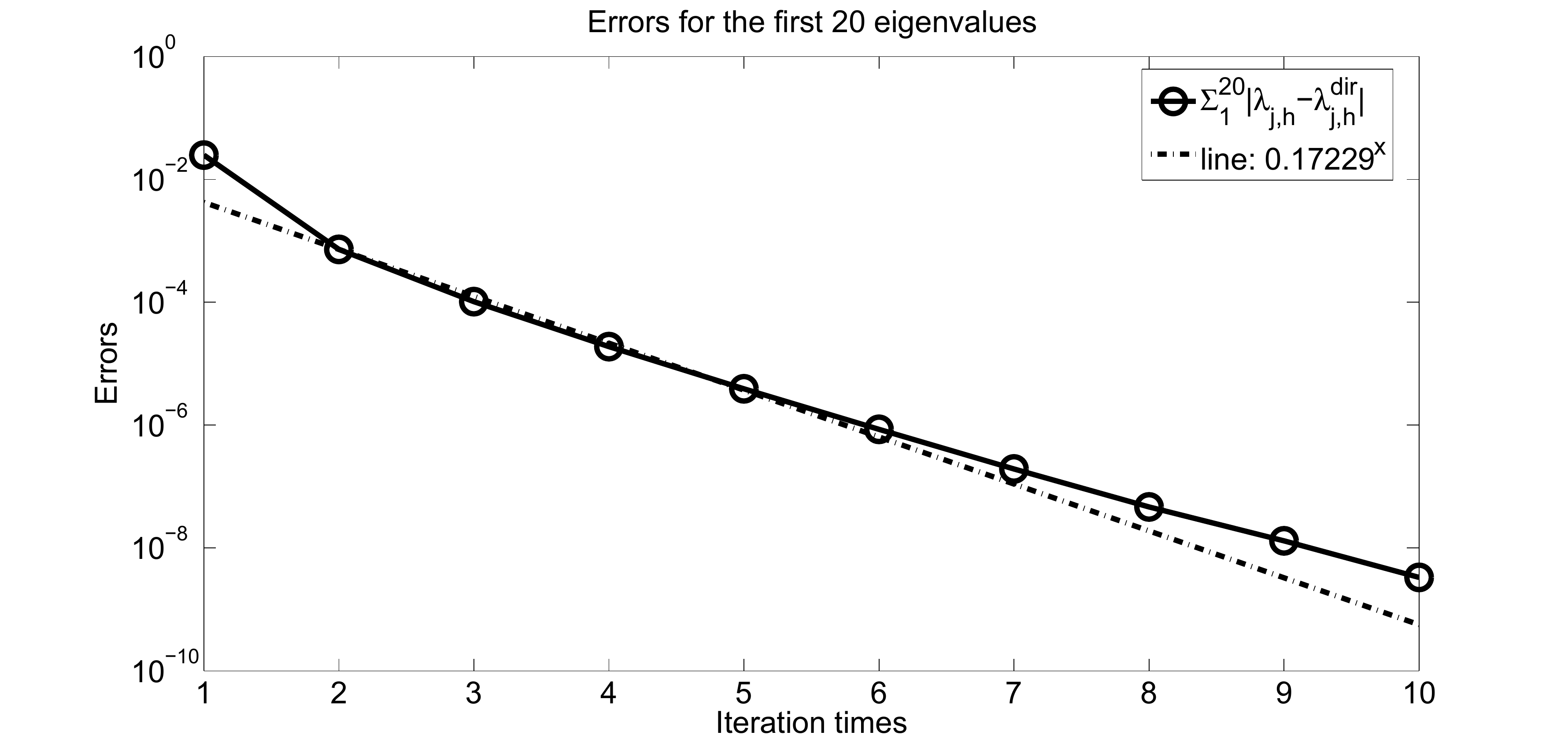}
\caption{Error estimates for the first $12$ eigenfunction and $20$ eigenvalue approximations by
Algorithm \ref{Multilevel_Correction_Eig}. Here the coarse mesh is chosen as the left one in Figure \ref{Exam_2_Coarse_Meshes}.}
\label{Exam_2_Error_H20}
\end{figure}
From Figures \ref{Exam_2_Error_H20} and \ref{Exam_2_Error_H32}, we can also find that
finer mesh $\mathcal T_H$ can lead to faster convergence speed which validates
Theorem \ref{Error_Estimate_One_Smoothing_Theorem}
and  Corollary \ref{Error_Estimate_Corollary}.
\begin{figure}[!hbt]
\centering
\includegraphics[width=7cm,height=5cm]{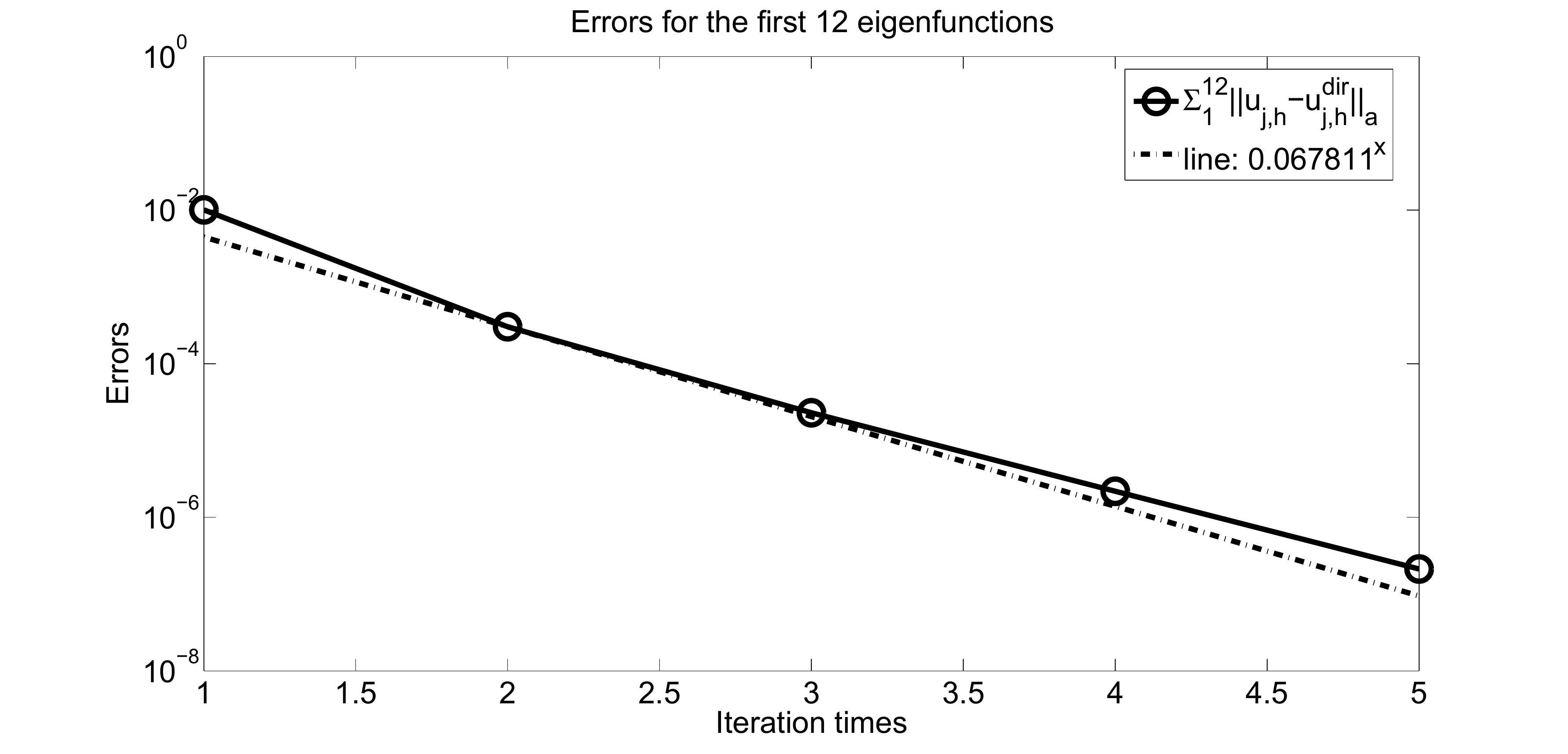}
\includegraphics[width=7cm,height=5cm]{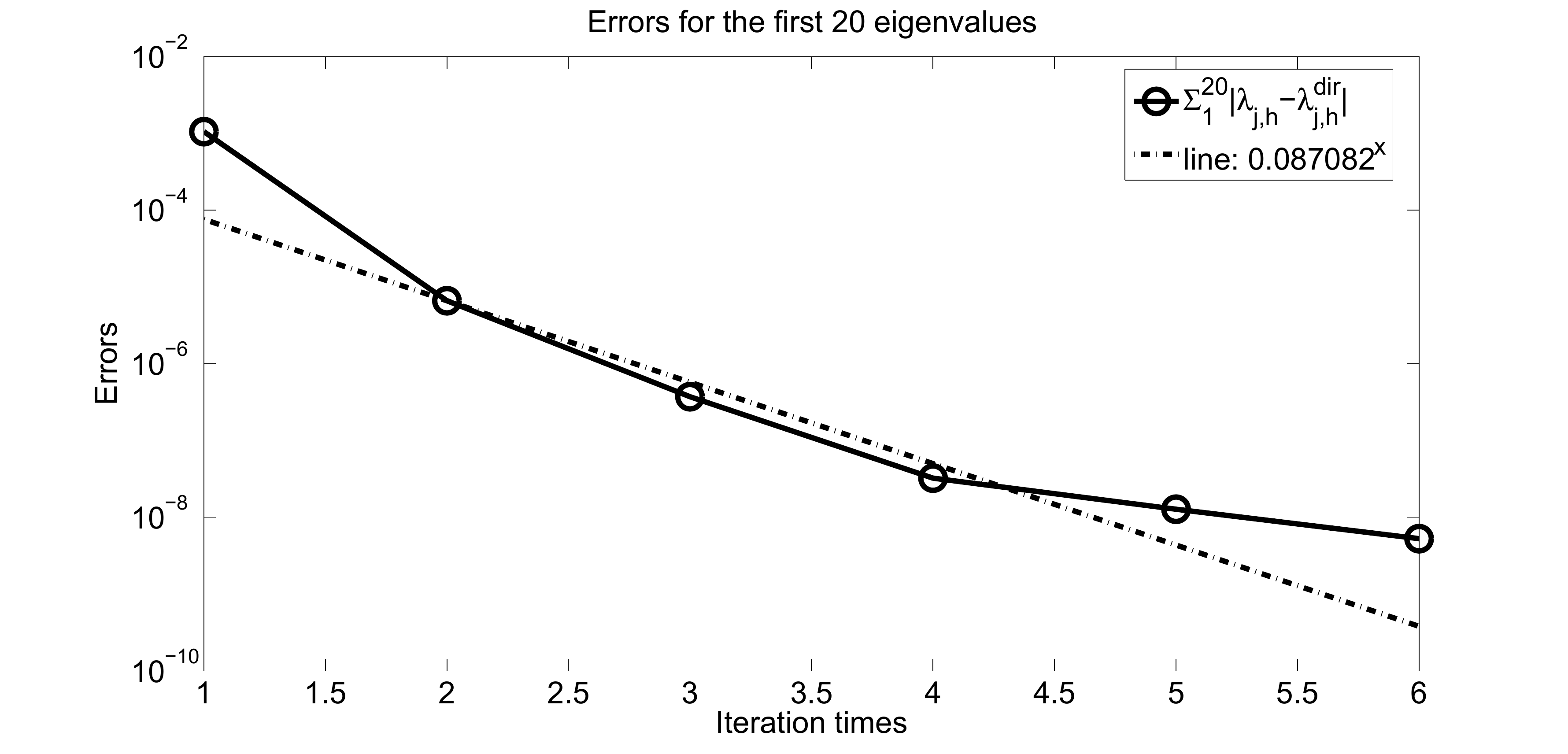}
\caption{Error estimates for the first $12$ eigenfunction and $20$ eigenvalue approximations by
Algorithm \ref{Multilevel_Correction_Eig}. Here the coarse mesh is chosen as the right one in Figure \ref{Exam_2_Coarse_Meshes}.}
\label{Exam_2_Error_H32}
\end{figure}

Similarly, we also investigate the efficiency with the CPU time for computing the first $10$ eigenpair approximations.
Here the convergence criterion is set to be $|\lambda_h-\bar \lambda_h|<1$e$-8$.
Figure \ref{Exam_2_CPUTime} shows the corresponding CPU time when the coarse meshes are chosen as the two in Figure \ref{Exam_2_Coarse_Meshes}
and the results also validate Theorem \ref{Work_Estimate_Multi_Correction}.
\begin{figure}[!hbt]
\centering
\includegraphics[width=7cm,height=5cm]{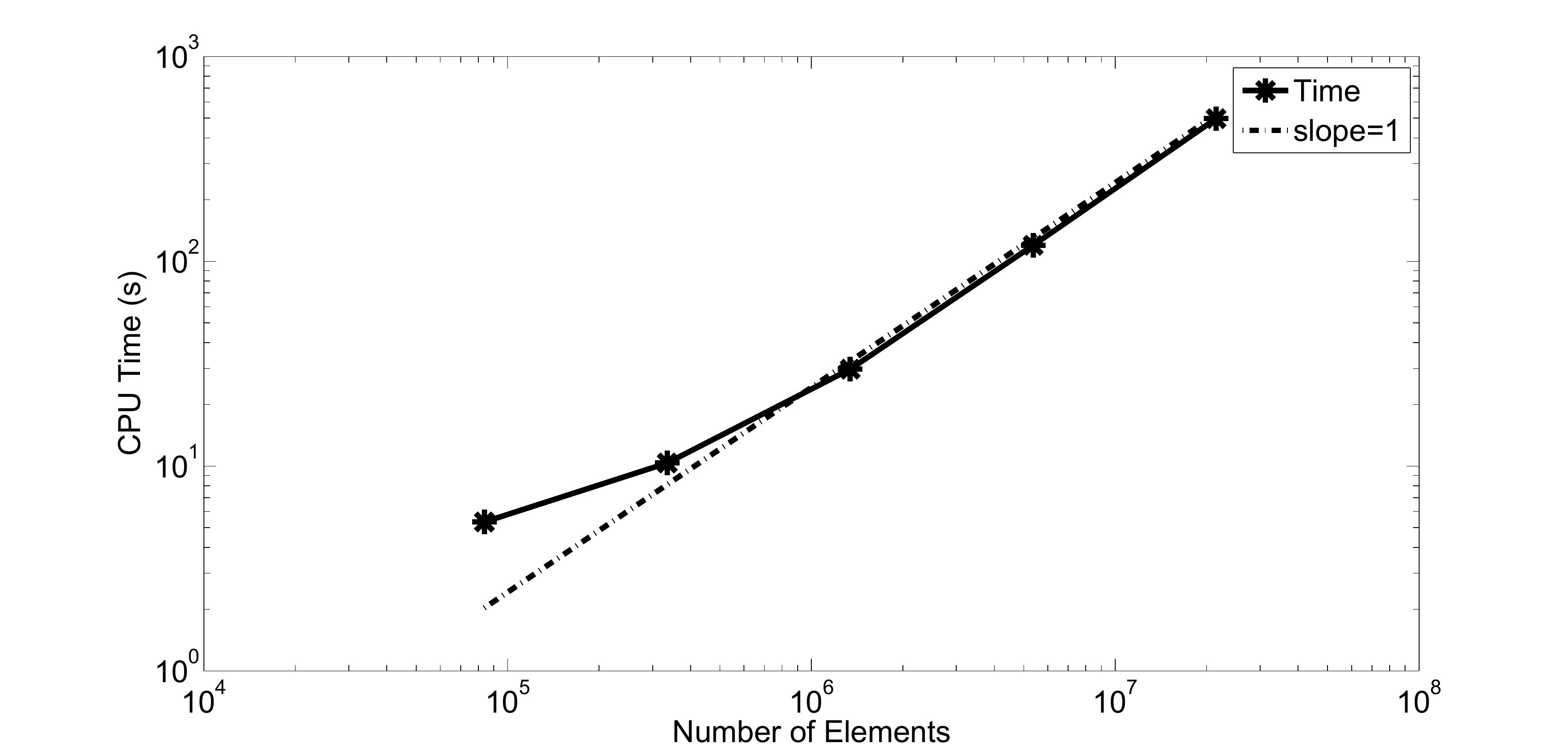}
\includegraphics[width=7cm,height=5cm]{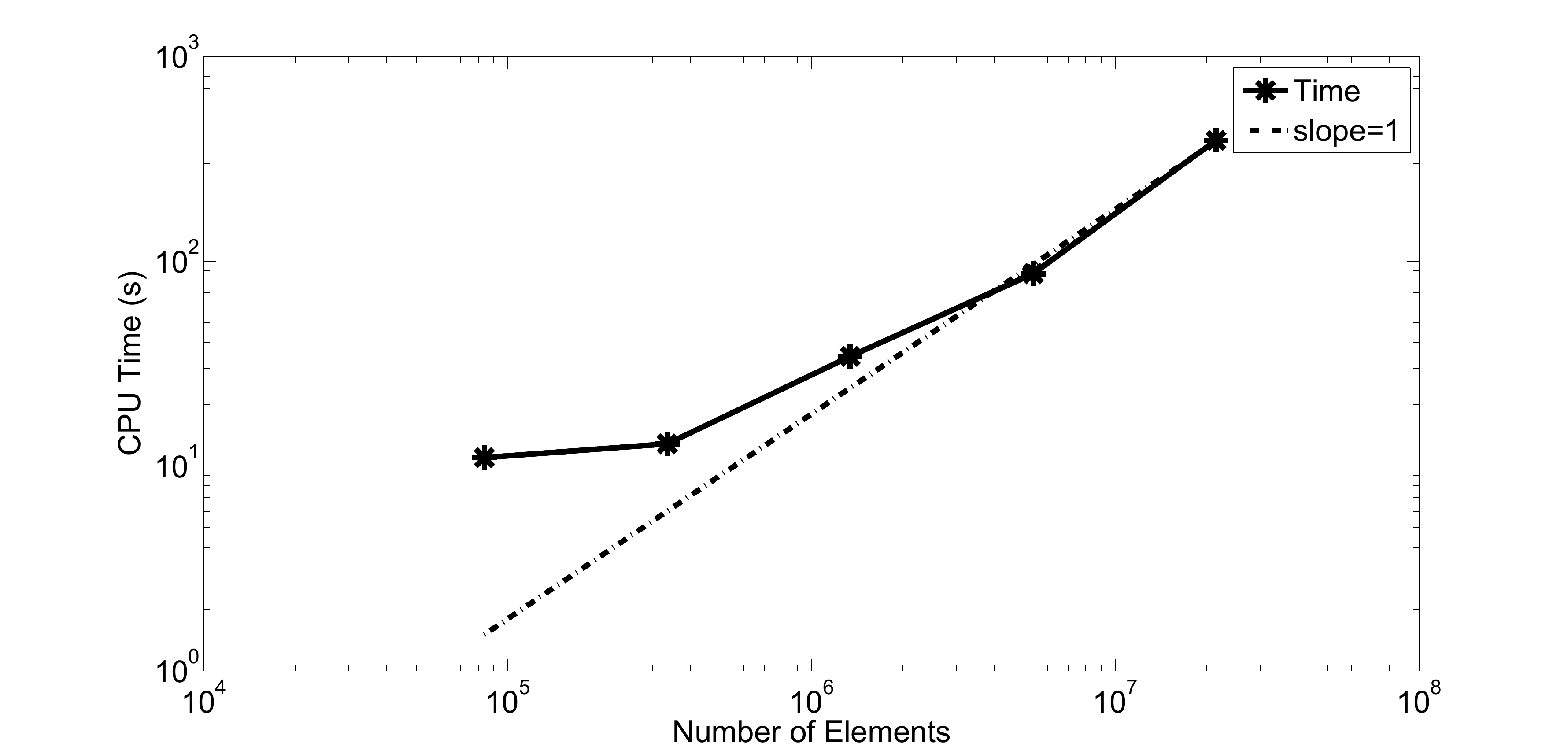}
\caption{CPU time for Algorithm \ref{Multilevel_Correction_Eig} \revise{with $32$ processors}, the left subfigure shows the CPU time when 
the coarse mesh is chosen as the left one in  Figure \ref{Exam_2_Coarse_Meshes} and the right subfigure shows the CPU time when the coarse 
mesh is chosen as the right one in Figure \ref{Exam_2_Coarse_Meshes}.}
\label{Exam_2_CPUTime}
\end{figure}

\subsection{Three dimensional experiments}
In the second subsection, the convergence and efficiency of Algorithms \ref{one correction step_Eig} and \ref{Multilevel_Correction_Eig}
are investigated for computing three dimensional eigenvalue problems.

\subsubsection*{Example 3}
In this example, we consider the elliptic eigenvalue problem (\ref{Eigenvalue_Problem_2d})
with a piecewise constant coefficient on the three dimensional domain $\Omega$ which includes a spherical surface interface.
The computing domain $\Omega = (0,2)\times (0,2)\times (0,2)$ is divided into two parts
by the surface of the sphere $\Omega_1$ with center $(1,1,1)$ and radius $0.5$.
Here, the coefficient  $\mathcal K$ is defined as follows
\begin{eqnarray}
\mathcal K = \left\{
\begin{array}{ll}
1,  & \textrm{in}\ \Omega_1=\{(x,y,z)\in\mathbb R^3| (x-1)^2+(y-1)^2+(z-1)^2\leq 1/4\},\\
10, & \textrm{in}\ \Omega_2=\Omega/\bar\Omega_1.
\end{array}
\right.
\end{eqnarray}

Similarly, in order to investigate the effect of the coarse grid $\mathcal T_H$ on the convergence behavior,
this example also selects two coarse meshes as shown in Figure \ref{Exam_3_Coarse_Meshes}.
For comparison, we use the same finest mesh with $650145$ elements for our test in this example.
Here, we check the convergence for the first $4$ eigenfunction and $10$ eigenvalue approximations.
\begin{figure}[!hbt]
\centering
\includegraphics[width=6cm,height=7cm]{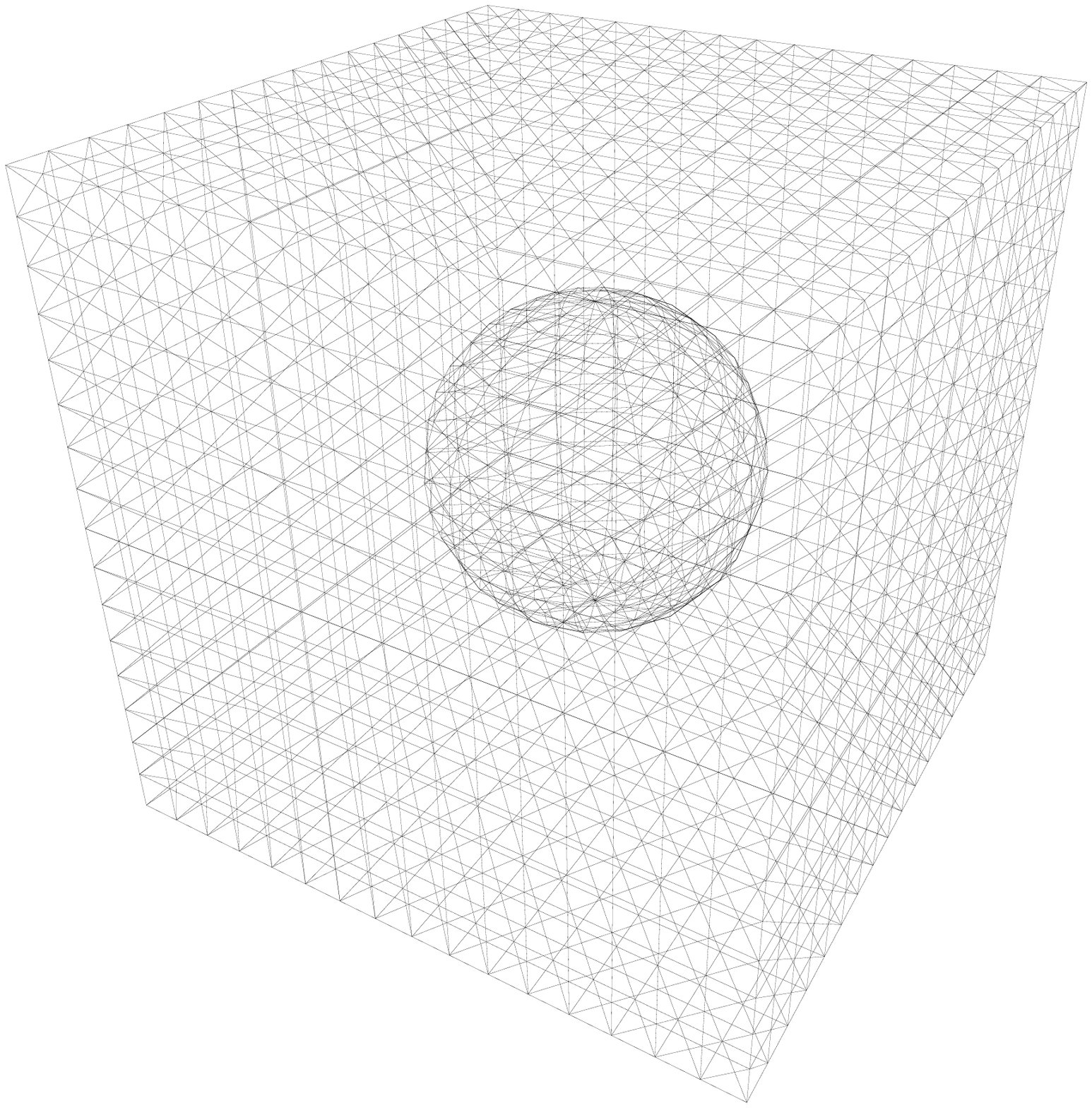}
\includegraphics[width=6cm,height=7cm]{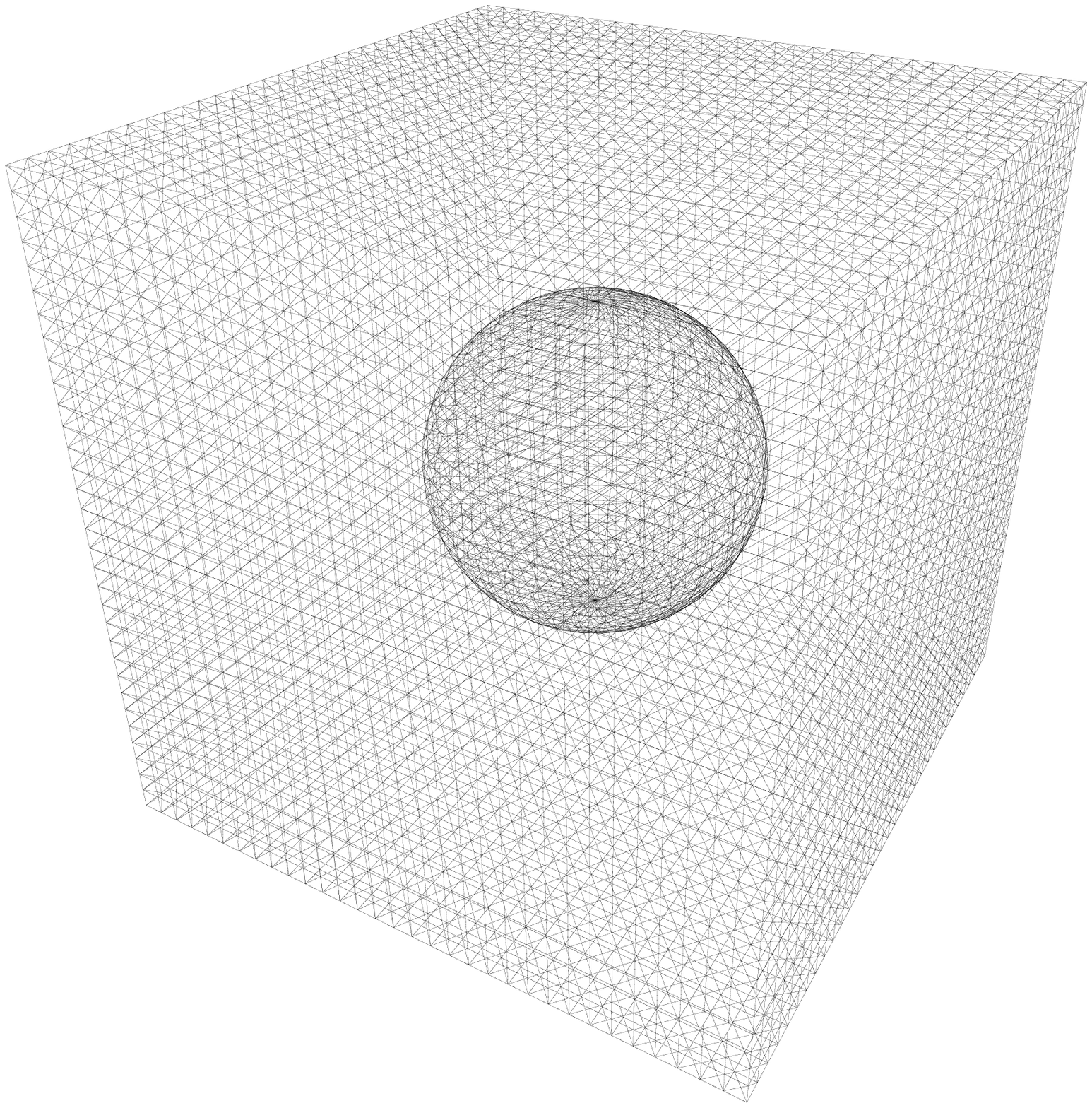}
\caption{Two coarse meshes $\mathcal T_H$ for Example 3: The left coarse mesh consists of $13169$ elements, and the right
one $40748$ elements.}\label{Exam_3_Coarse_Meshes}
\end{figure}
When the coarse meshes $\mathcal T_H$ are chosen as the left one and right one in Figure \ref{Exam_3_Coarse_Meshes},
the corresponding numerical results are shown in Figures \ref{Exam_3_Error_H16} and \ref{Exam_3_Error_H32}, respectively.
From Figures \ref{Exam_3_Error_H16} and \ref{Exam_3_Error_H32}, we can find that the finer $\mathcal T_H$ has better
convergence rate, which validates the theoretical results in Theorem \ref{Error_Estimate_One_Smoothing_Theorem} and Corollary \ref{Error_Estimate_Corollary}.
\begin{figure}[!hbt]
\centering
\includegraphics[width=7cm,height=5cm]{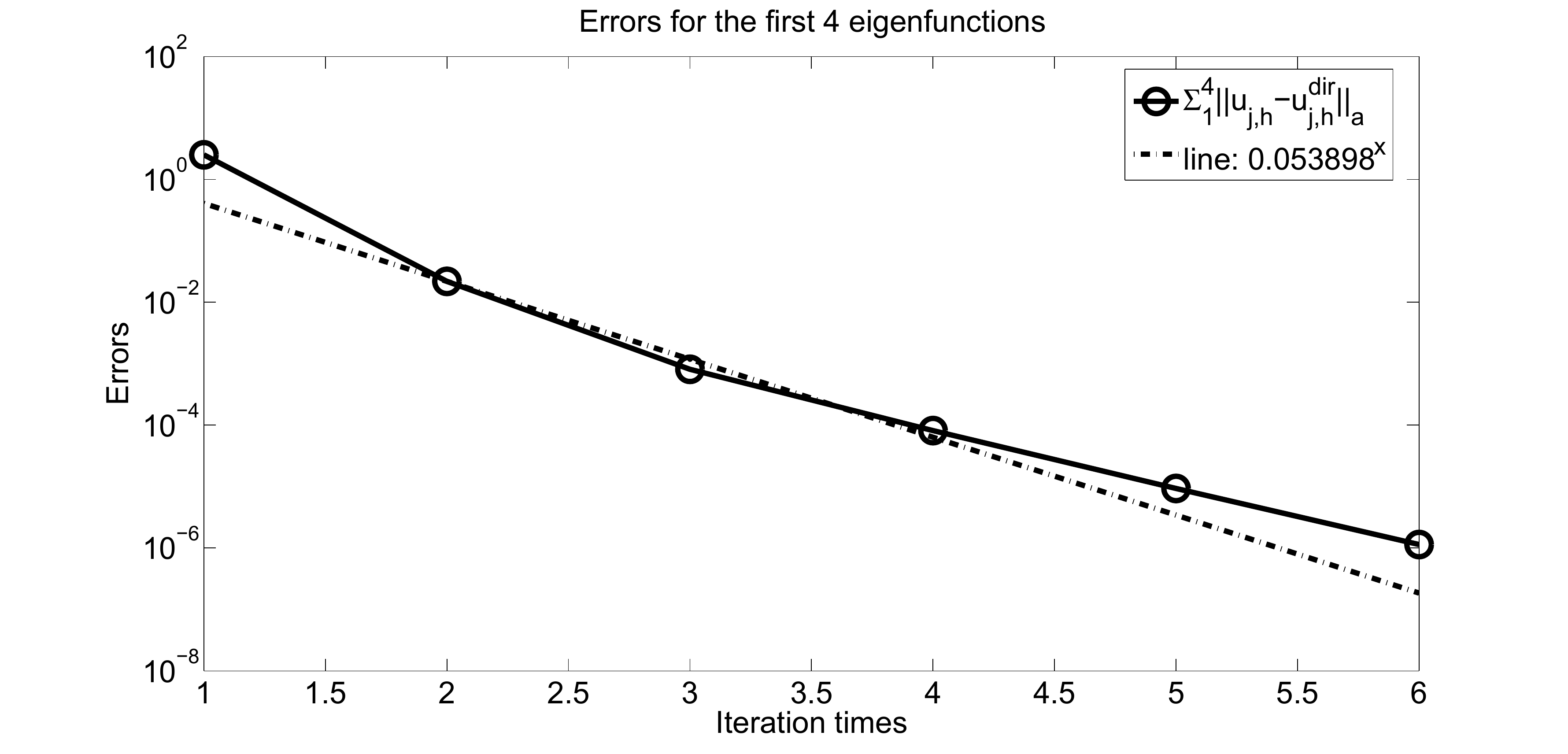}
\includegraphics[width=7cm,height=5cm]{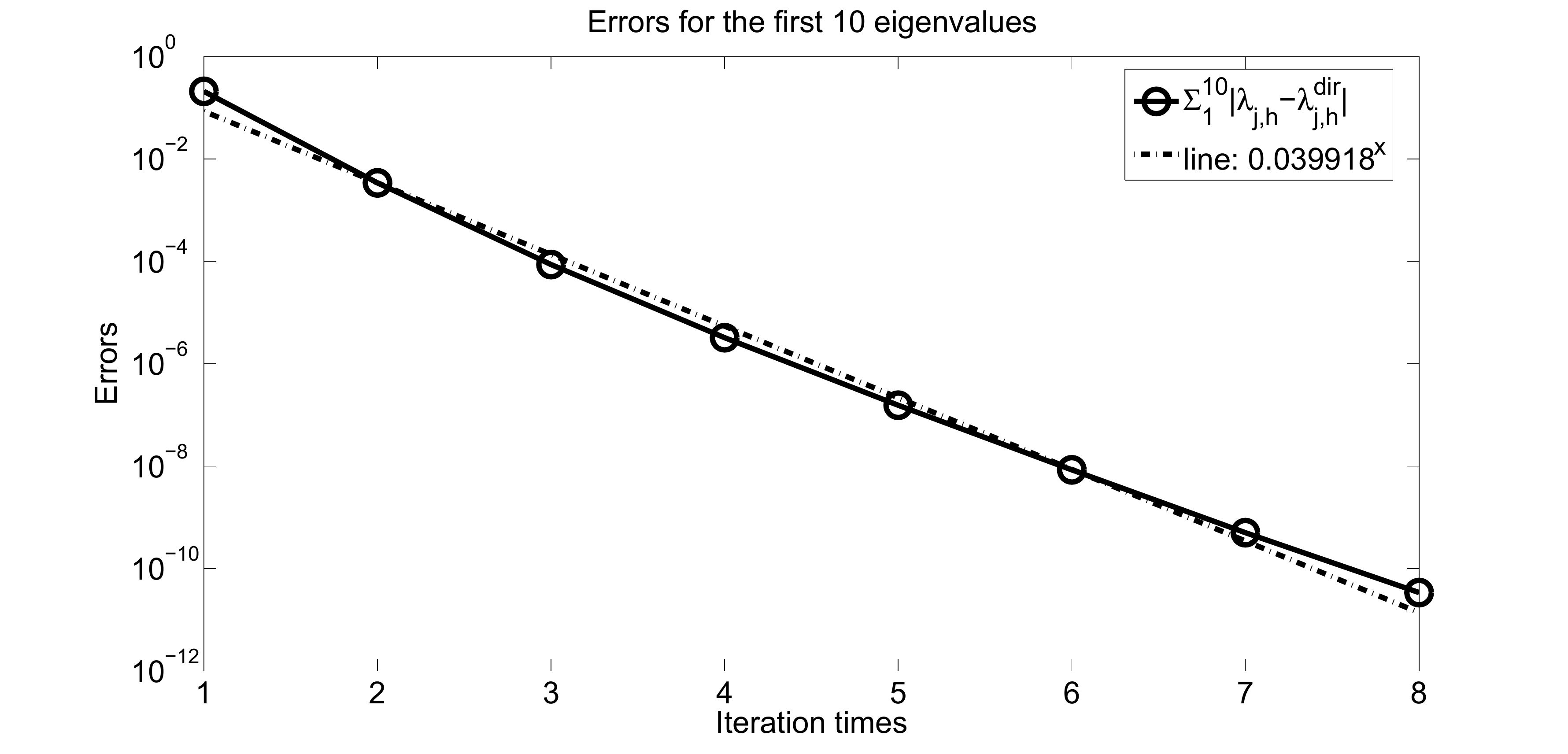}
\caption{Error estimates for the first $4$ eigenfunction and $10$ eigenvalue approximations by
Algorithm \ref{Multilevel_Correction_Eig}. Here the coarse mesh is chosen as the left one in Figure \ref{Exam_3_Coarse_Meshes}.}
\label{Exam_3_Error_H16}
\end{figure}
\begin{figure}[!hbt]
\centering
\includegraphics[width=7cm,height=5cm]{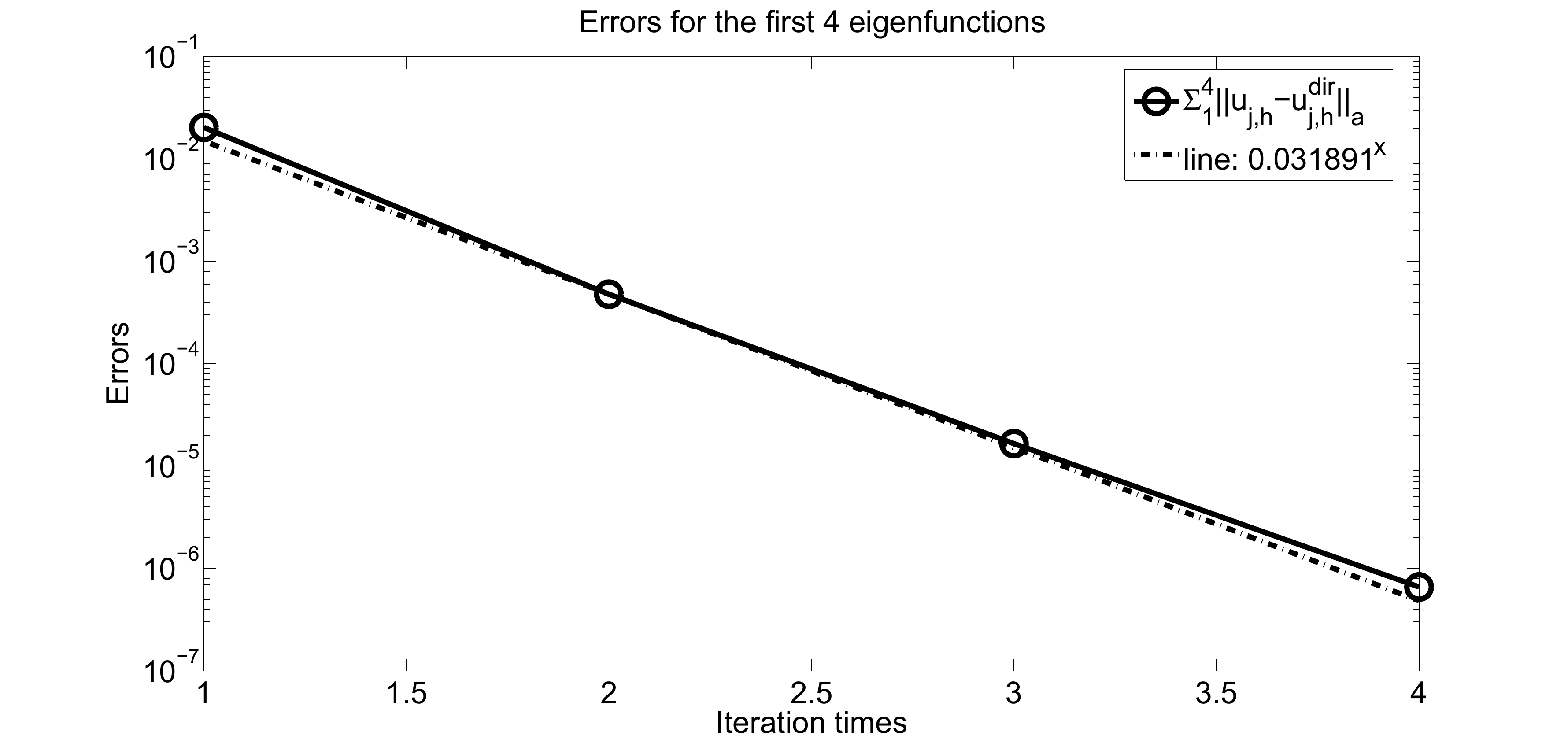}
\includegraphics[width=7cm,height=5cm]{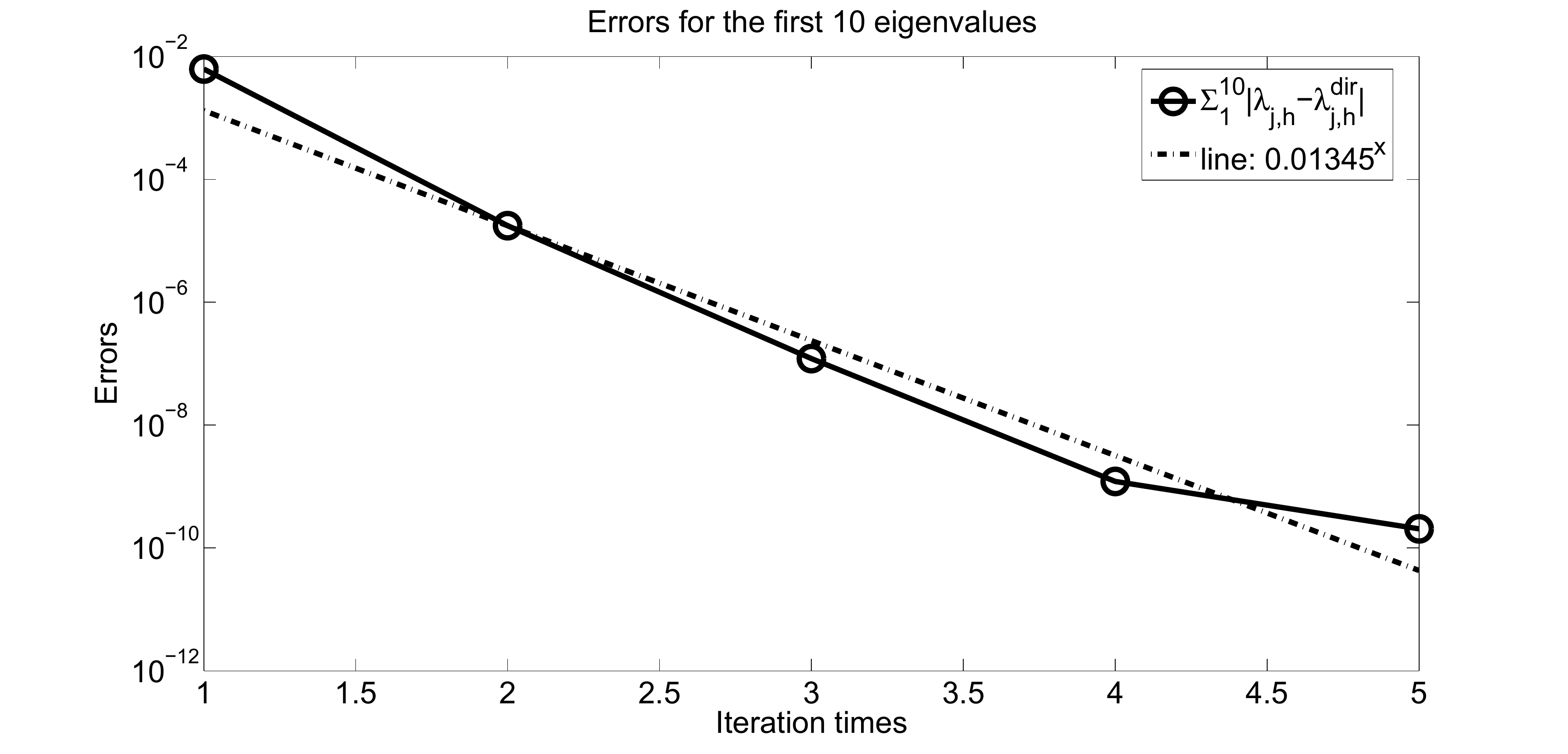}
\caption{Error estimates for the first $4$ eigenfunction and $10$ eigenvalue approximations by
Algorithm \ref{Multilevel_Correction_Eig}. Here the coarse mesh is chosen as the right one in Figure \ref{Exam_3_Coarse_Meshes}.}
\label{Exam_3_Error_H32}
\end{figure}

In order to check the efficiency of the proposed algorithms in this paper, we also check the CPU time for computing the first $10$
eigenpair approximations. The convergence criterion is set to be $|\lambda_h-\bar \lambda_h|<1\textrm{e}$-$9$.
Figure \ref{Exam_3_CPUTime} shows the CPU time results corresponding to the two coarse meshes in Figure \ref{Exam_3_Coarse_Meshes}.
The results here also show the linear scale of the complexity for Algorithm \ref{Multilevel_Correction_Eig}
for the three dimensional eigenvalue problems with curved interfaces.
\begin{figure}[!hbt]
\centering
\includegraphics[width=7cm,height=5cm]{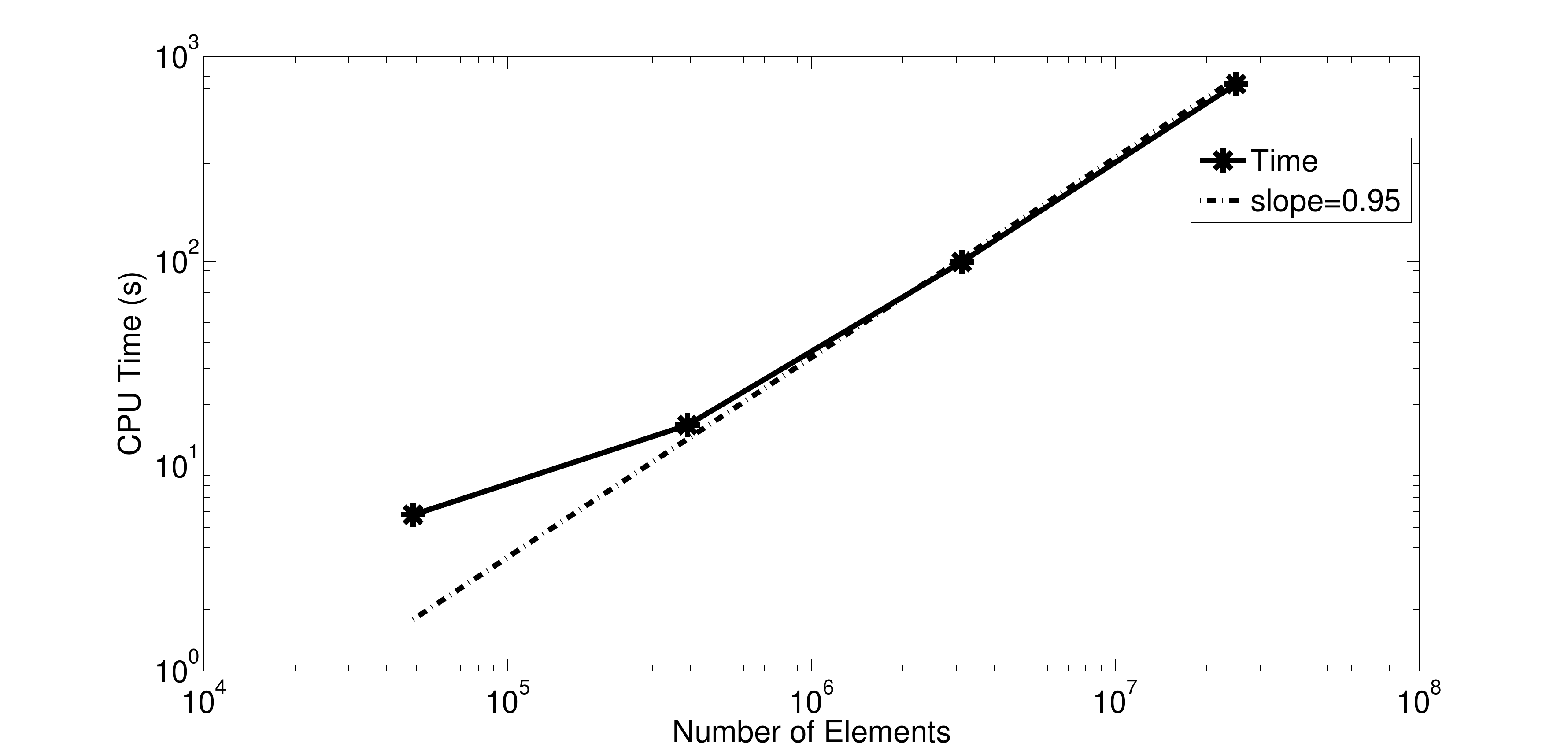}%Exam_3_10_CPUTime_16_np16
\includegraphics[width=7cm,height=5cm]{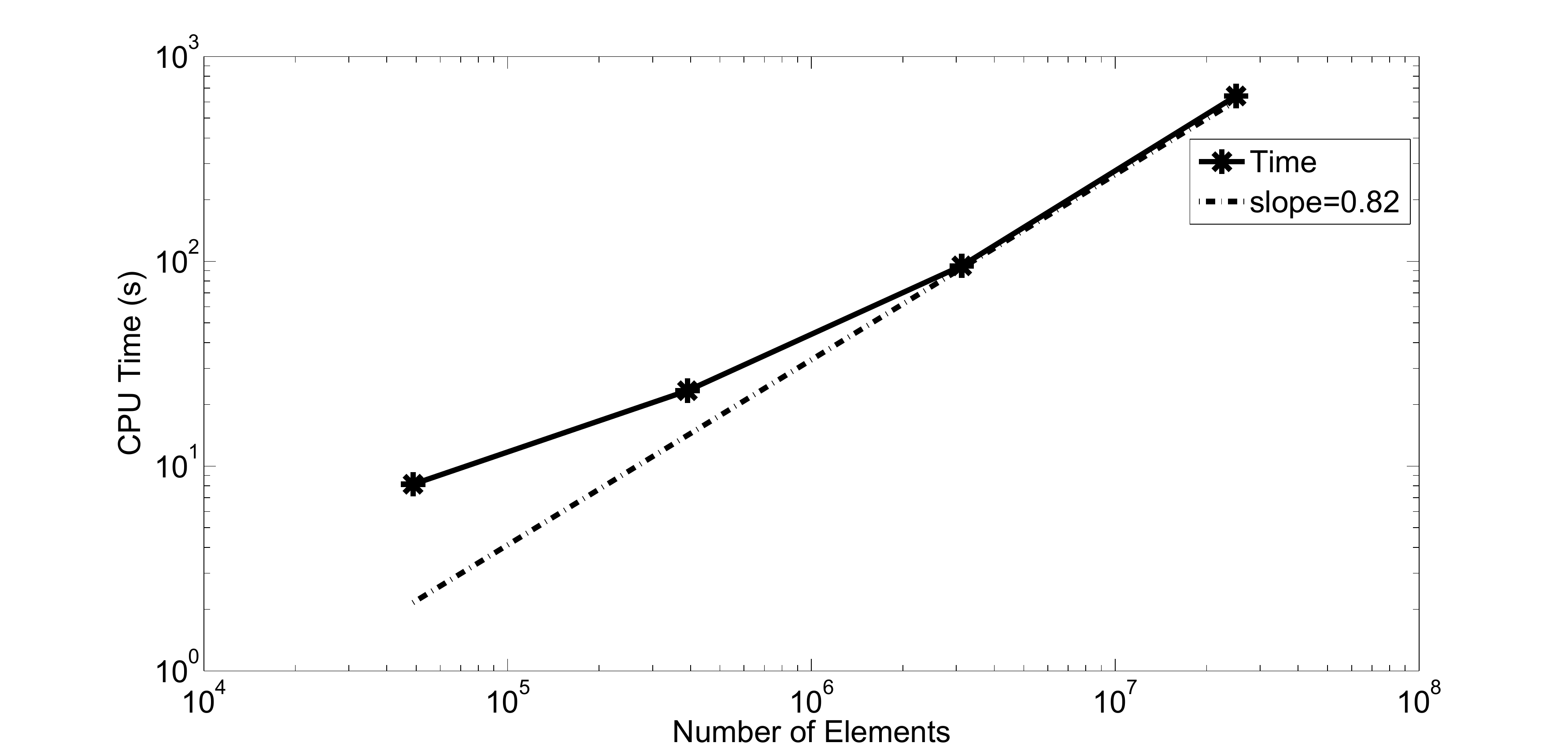}
\caption{CPU time for Algorithm \ref{Multilevel_Correction_Eig} \revise{with $16$ processors}, the left subfigure
shows the CPU time when the coarse mesh is chosen as the left one in  Figure \ref{Exam_3_Coarse_Meshes}
and the right subfigure shows the CPU time when the coarse mesh is chosen
as the right one in Figure \ref{Exam_3_Coarse_Meshes}.}
\label{Exam_3_CPUTime}
\end{figure}

\subsubsection*{Example 4}
In this example, we focus on the three-dimensional elliptic eigenvalue problem (\ref{Eigenvalue_Problem_2d})
with a piecewise constant coefficient which is defined on the three-dimensional domain $\Omega = (0,2)\times (0,2)\times (0,2)$
with curve interfaces by two spheres. The computing domain is partitioned into three parts by two spheres with
centers  $(0.5,0.5,0.5)$ and $(1.5, 1.5, 1.5)$ and radius sizes $1.3$ and $1/3$, respectively.
The coefficient $\mathcal K$ is defined as follows
\begin{eqnarray}
\mathcal K = \left\{
\begin{array}{ll}
10, & \textrm{in}\  \Omega_1=\{(x,y,z)\in\mathbb R^3| (x-0.5)^2+(y-0.5)^2+(z-0.5)^2\leq 1/9\},\\
10, & \textrm{in}\  \Omega_2=\{(x,y,z)\in\mathbb R^3| (x-1.5)^2+(y-1.5)^2+(z-1.5)^2\leq 1/9\},\\
1, & \textrm{in}\ \Omega/(\bar\Omega_1\cup\bar\Omega_2).
\end{array}
\right.
\end{eqnarray}
Here, we also select two coarse meshes as shown in Figure \ref{Exam_4_Coarse_Meshes} for our tests.
For comparison, we use the same finest mesh with  $1178024$ elements for checking the convergence behaviors.

\begin{figure}[!hbt]
\centering
\includegraphics[width=6cm,height=7.5cm]{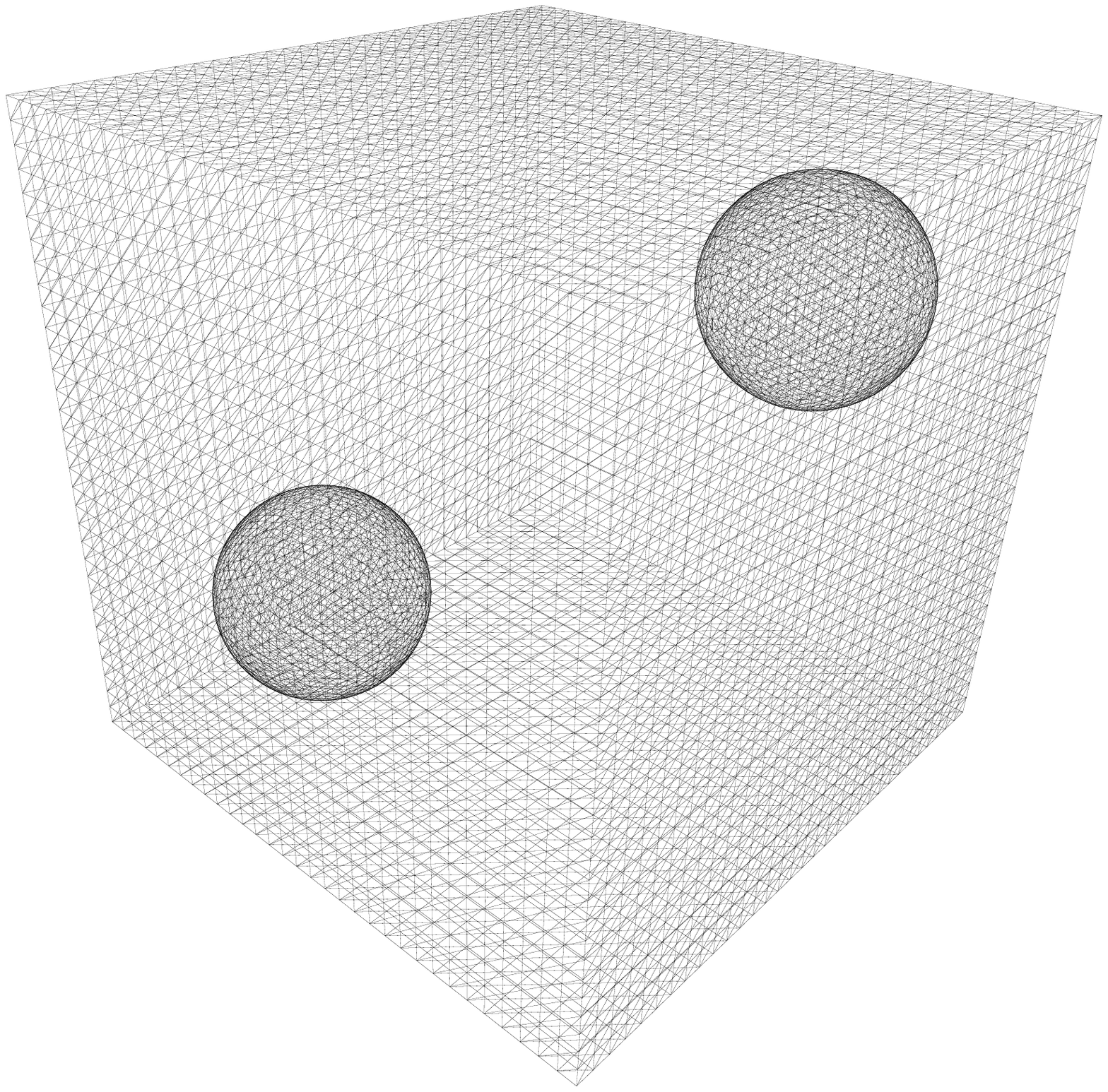}
\includegraphics[width=6cm,height=7.5cm]{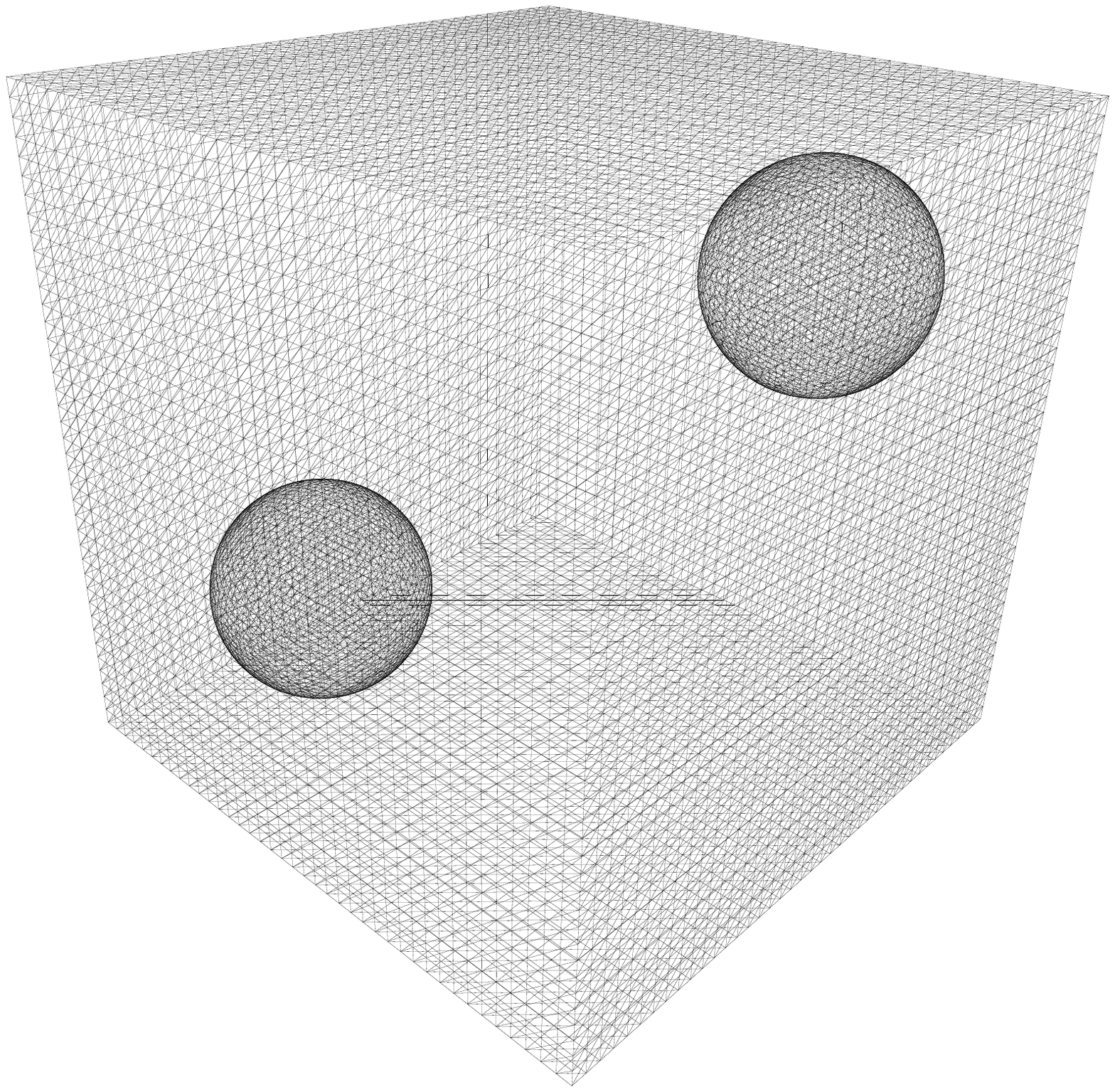}
\caption{Two coarse meshes $\mathcal T_H$ for Example 4: The left coarse mesh consists of $56660$ elements, and the right
one $92453$ elements.}\label{Exam_4_Coarse_Meshes}
\end{figure}

Figure \ref{Exam_4_Error_H32} and \ref{Exam_4_Error_H40} show the numerical results for the first $4$ eigenfunction and
$11$ eigenvalue approximations when the coarse meshes are chosen as the left and right ones in Figure \ref{Exam_4_Coarse_Meshes}.
From these two figures, we can also find the finer $\mathcal T_H$ leads to faster convergence speed which confirm the theoretical
results in Theorems \ref{Error_Estimate_One_Smoothing_Theorem} and \ref{Error_Multi_Correction_Theorem},
Corollary \ref{Error_Estimate_Corollary}.
\begin{figure}[!hbt]
\centering
\includegraphics[width=7cm,height=5cm]{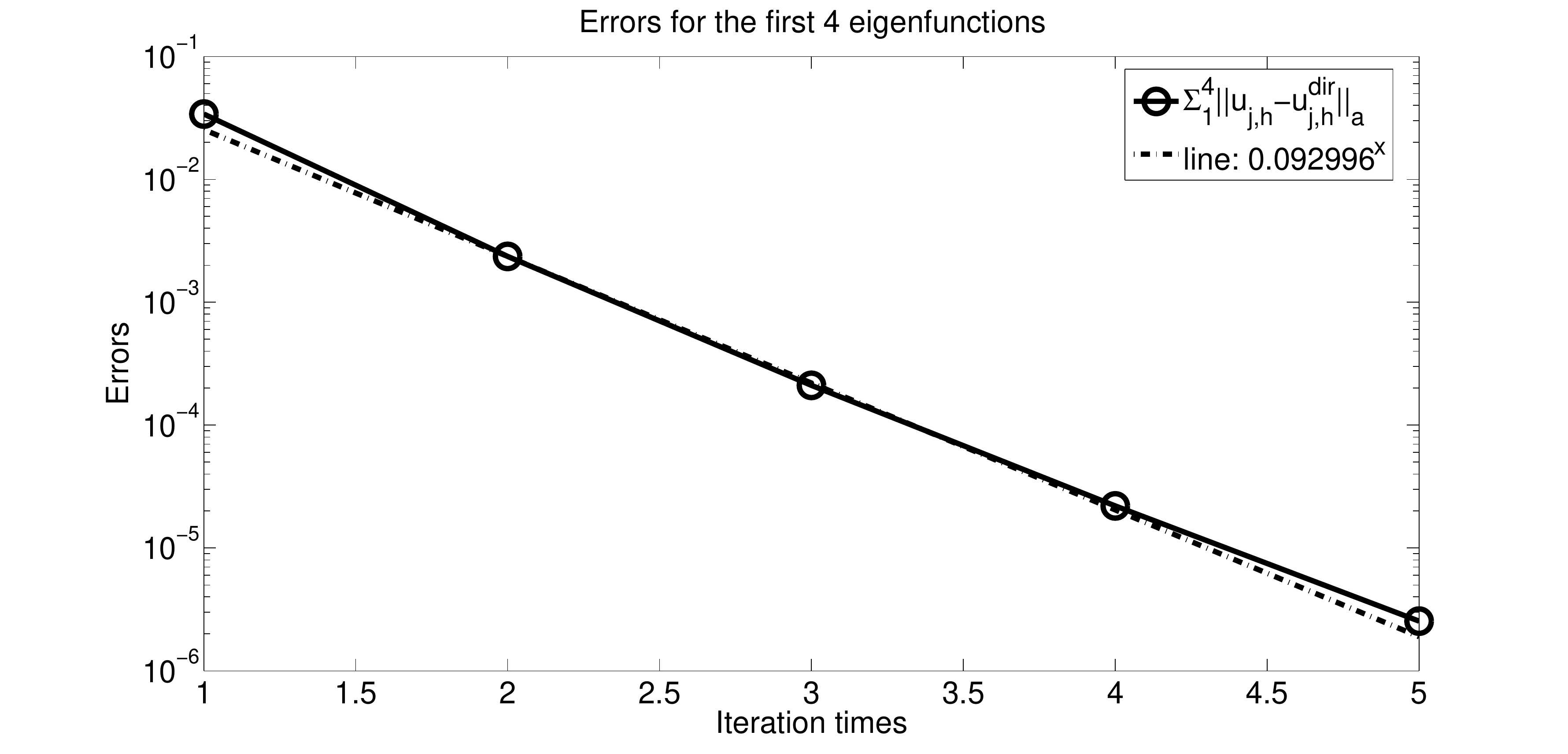}
\includegraphics[width=7cm,height=5cm]{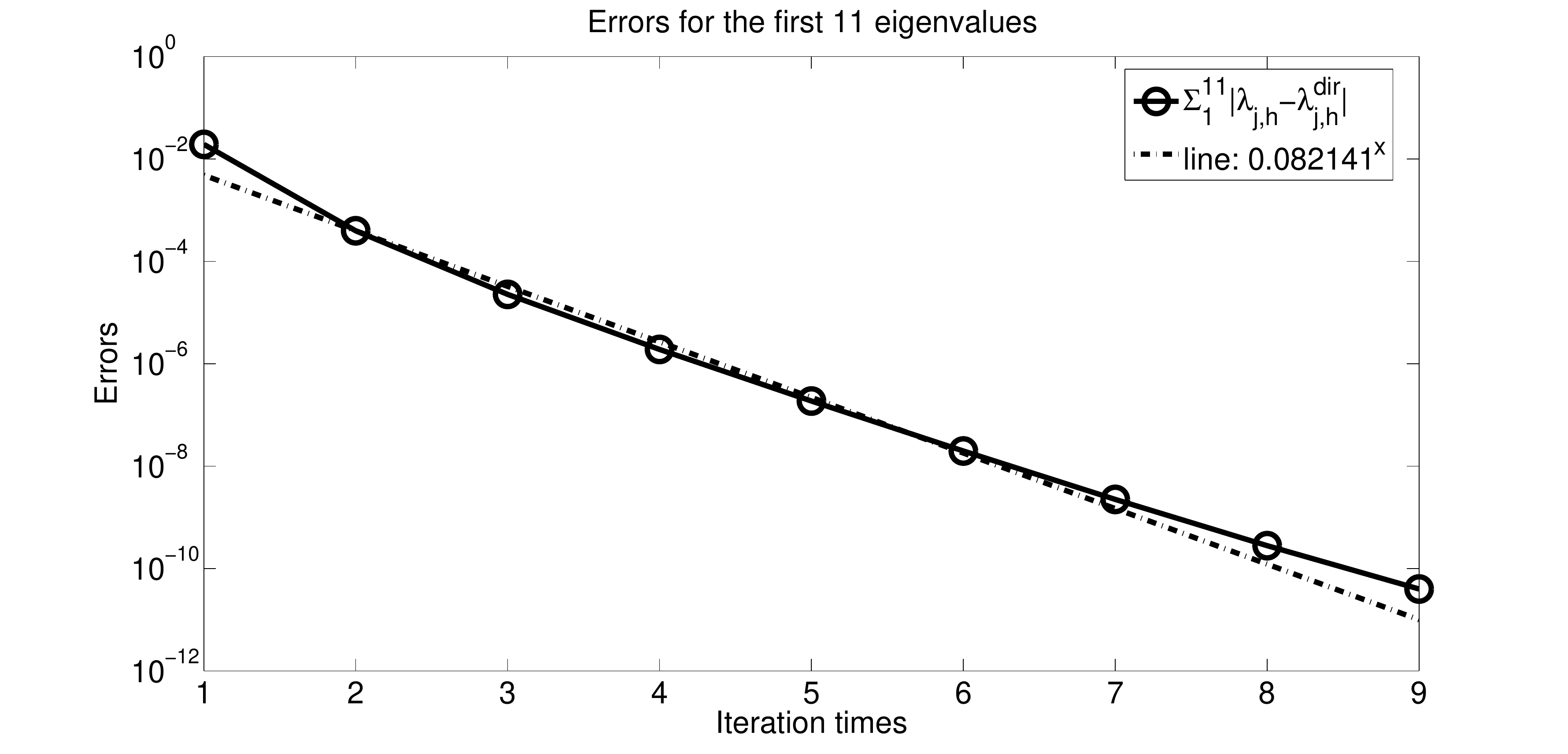}
\caption{Error estimates for the first $4$ eigenfunction and $11$ eigenvalue approximations by
Algorithm \ref{Multilevel_Correction_Eig}. Here the coarse mesh is chosen as the left one in Figure \ref{Exam_4_Coarse_Meshes}.}
\label{Exam_4_Error_H32}
\end{figure}
\begin{figure}[!hbt]
\centering
\includegraphics[width=7cm,height=5cm]{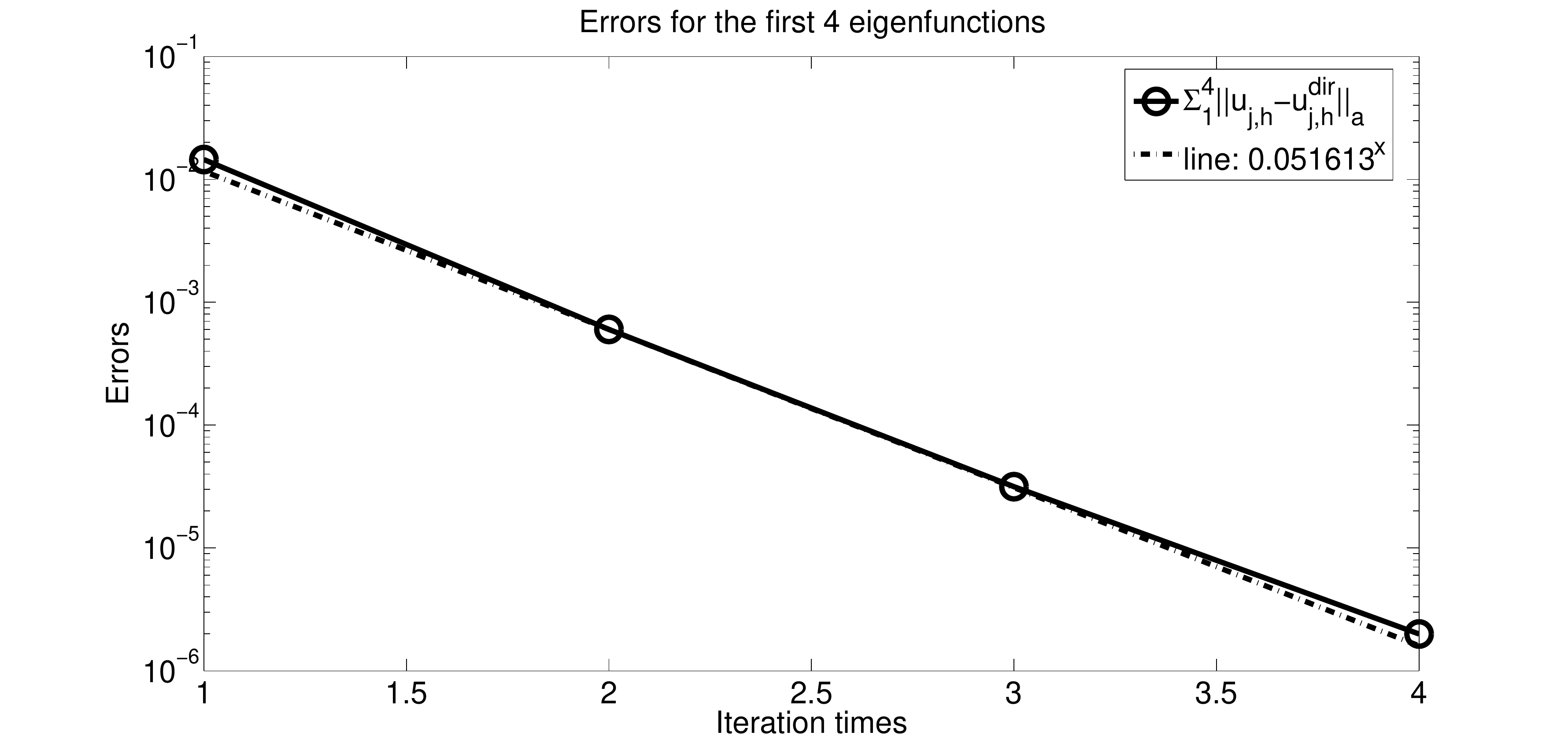}
\includegraphics[width=7cm,height=5cm]{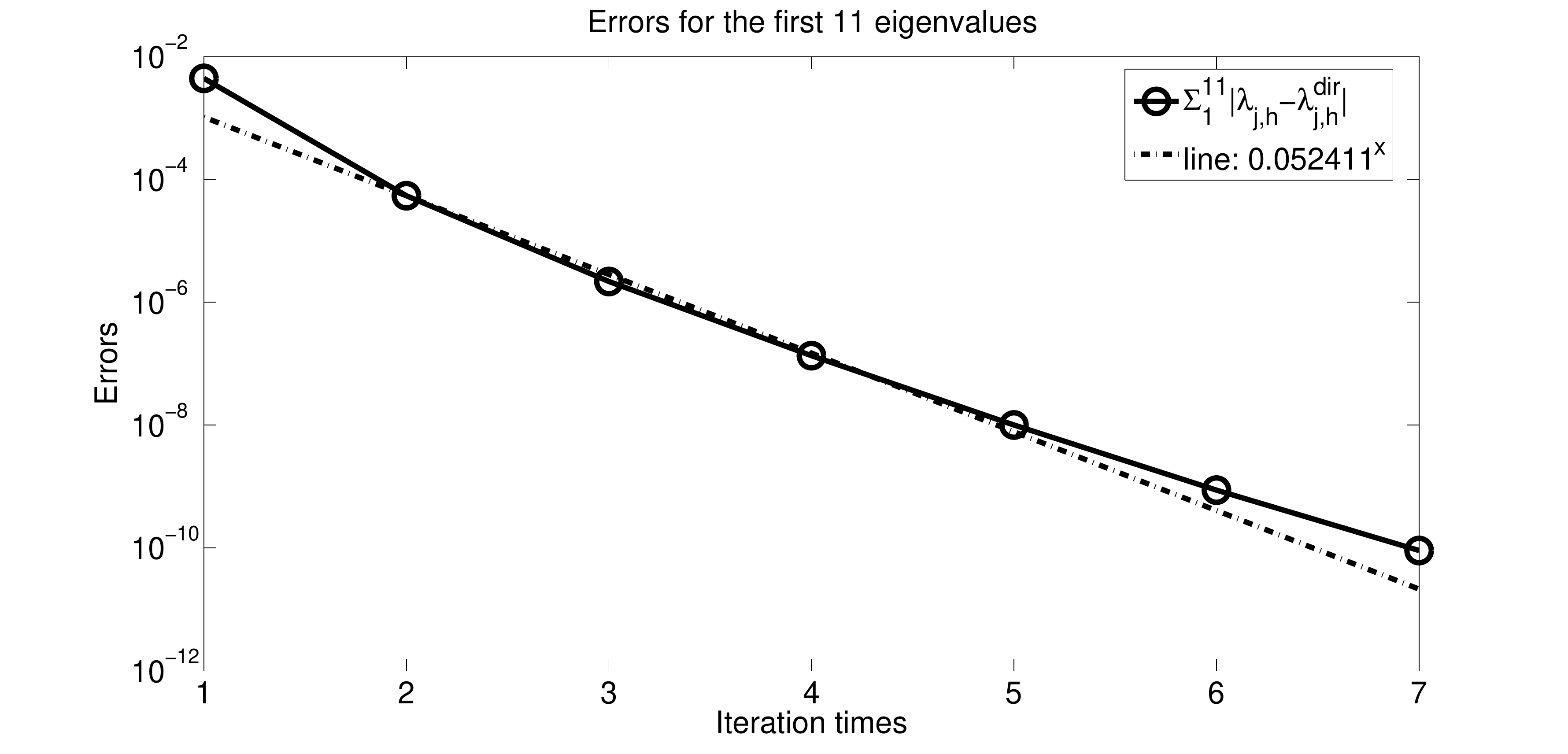}
\caption{Error estimates for the first $4$ eigenfunction and $11$ eigenvalue approximations by
Algorithm \ref{Multilevel_Correction_Eig}. Here the coarse mesh is chosen as the right one in Figure \ref{Exam_4_Coarse_Meshes}.}
\label{Exam_4_Error_H40}
\end{figure}

Here, we also present the CPU time results for computing the first $11$ eigenpair approximations.
The convergence criterion is also set to be $|\lambda_h-\bar\lambda_h|<1\textrm{e}$-$9$.
Figure \ref{Exam_4_CPUTime} shows the CPU time results corresponding to the two coarse meshes in Figure \ref{Exam_4_Coarse_Meshes}.
The results here also show the linear scale of the complexity for Algorithm \ref{Multilevel_Correction_Eig}
for the three dimensional eigenvalue problems with curved interfaces.
\begin{figure}[!hbt]
\centering
\includegraphics[width=7cm,height=5cm]{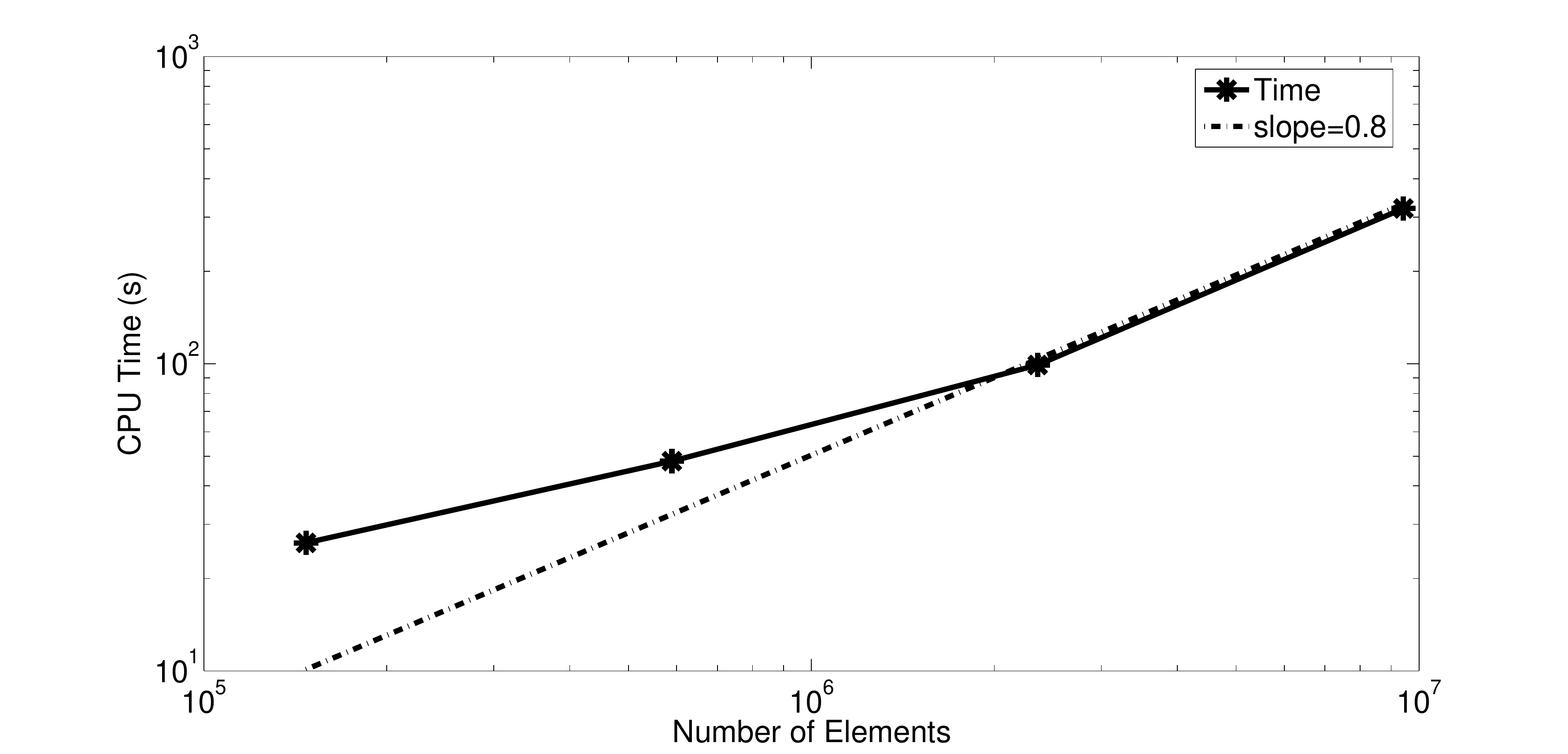}
\includegraphics[width=7cm,height=5cm]{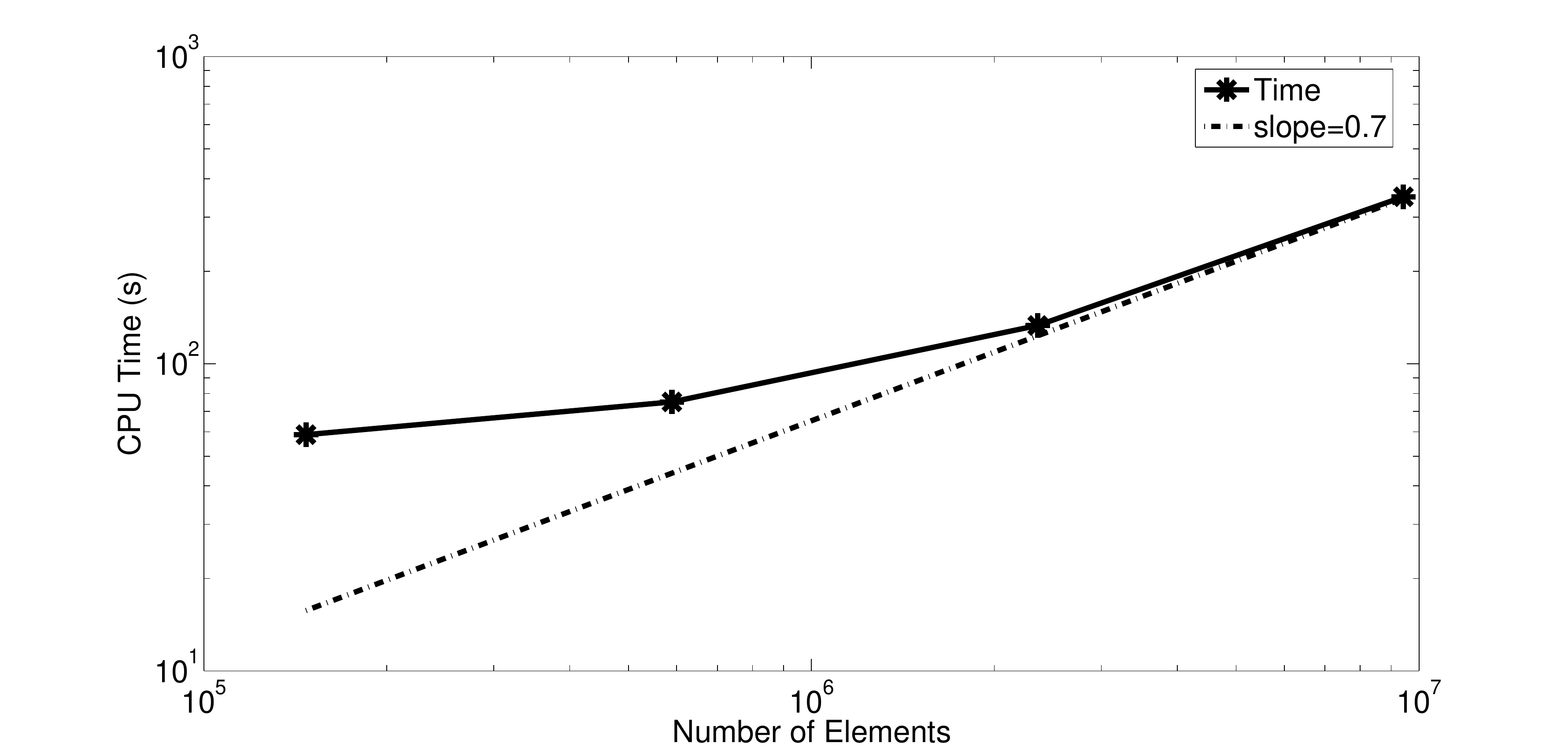}
\caption{CPU time for Algorithm \ref{Multilevel_Correction_Eig} \revise{with $16$ processors},
the left subfigure shows the CPU time when the coarse mesh is chosen as the left one in
Figure \ref{Exam_4_Coarse_Meshes} and the right subfigure shows the CPU time
when the coarse mesh is chosen as the right one in Figure \ref{Exam_4_Coarse_Meshes}.}
\label{Exam_4_CPUTime}
\end{figure}

\section{Conclusions}
In this paper, we design a nonnested augmented subspace method and the corresponding
multilevel correction scheme for solving eigenvalue problems with curved interfaces.
Throughout this paper, we demonstrate that the augmented subspace method can also work on the nonnested
sequence of meshes. The proposed algorithms here provide a way to combine the augmented subspace method (multilevel correction method)
with the moving mesh techniques. This will improve the overall efficiency for solving the eigenvalue problems with anisotropy
and singularity. The method in this paper can be extended to nonlinear eigenvalue problems and this will be our future work.

\section*{Acknowledgments}
We are very grateful to Prof. Pierre Jolivet for his kind discussion and help to implement numerical examples with FreeFEM++.
Especially, Prof. Pierre Jolivet help us to do the efficient interpolation between two nonnested meshes
which is very important for implementing the proposed method in this paper.
Here, we express our thanks for all developers of FreeFEM++.

%%%\bibliographystyle{unsrt}
%\bibliographystyle{siam}
%%%\bibliographystyle{plainnat}
%%%\bibliographystyle{elsarticle-num}
%%%\bibliographystyle{elsarticle-harv}
%%%\bibliographystyle{elsarticle-num-names}
%\bibliography{mybibfile.bib}

\begin{thebibliography}{999}

\bibitem{Adams}
{ R.~A. Adams}, {\em Sobolev Spaces}, Academic Press [A subsidiary of
  Harcourt Brace Jovanovich, Publishers], New York-London, 1975.
\newblock Pure and Applied Mathematics, Vol. 65.

\bibitem{Babuska_1970}
{ I.~Babu\v{s}ka}, {\em The finite element method for elliptic equations
  with discontinuous coefficients}, Computing, 5 (1970), pp.~207--213.

\bibitem{BabuskaOsborn_1989}
{ I.~Babu\v{s}ka and J.~E. Osborn}, {\em Finite element-{G}alerkin
  approximation of the eigenvalues and eigenvectors of selfadjoint problems},
  Math. Comp., 52 (1989), pp.~275--297.

\bibitem{petsc-web-page}
{ S.~Balay, S.~Abhyankar, M.~F. Adams, J.~Brown, P.~Brune, K.~Buschelman,
  L.~Dalcin, A.~Dener, V.~Eijkhout, W.~D. Gropp, D.~Karpeyev, D.~Kaushik, M.~G.
  Knepley, D.~A. May, L.~C. McInnes, R.~T. Mills, T.~Munson, K.~Rupp, P.~Sanan,
  B.~F. Smith, S.~Zampini, H.~Zhang, and H.~Zhang}, {\em {PETS}c {W}eb page}.
\newblock \url{https://www.mcs.anl.gov/petsc}, 2019.

\bibitem{petsc-user-ref}
\leavevmode\vrule height 2pt depth -1.6pt width 23pt, {\em {PETS}c users
  manual}, Tech. Report ANL-95/11 - Revision 3.14, Argonne National Laboratory,
  2020.

\bibitem{petsc-efficient}
{ S.~Balay, W.~D. Gropp, L.~C. McInnes, and B.~F. Smith}, {\em Efficient
  management of parallelism in object oriented numerical software libraries},
  in Modern Software Tools in Scientific Computing, E.~Arge, A.~M. Bruaset, and
  H.~P. Langtangen, eds., Birkh{\"{a}}user Press, 1997, pp.~163--202.

\bibitem{BankDupont}
{ R.~E. Bank and T.~Dupont}, {\em An optimal order process for solving
  finite element equations}, Math. Comp., 36 (1981), pp.~35--51.

\bibitem{Bramble}
{ J.~H. Bramble}, {\em Multigrid Methods}, vol.~294 of Pitman Research Notes
  in Mathematics Series, Longman Scientific \& Technical, Harlow; copublished
  in the United States with John Wiley \& Sons, Inc., New York, 1993.

\bibitem{BramblePasciak}
{ J.~H. Bramble and J.~E. Pasciak}, {\em New convergence estimates for
  multigrid algorithms}, Math. Comp., 49 (1987), pp.~311--329.

\bibitem{BrambleZhang}
{ J.~H. Bramble and X.~Zhang}, {\em The analysis of multigrid methods}, in
  Handbook of Numerical Analysis, {V}ol. {VII}, Handb. Numer. Anal., VII,
  North-Holland, Amsterdam, 2000, pp.~173--415.

\bibitem{BrandtMcCormickRuge}
{ A.~Brandt, S.~McCormick, and J.~Ruge}, {\em Multigrid methods for
  differential eigenproblems}, SIAM J. Sci. Statist. Comput., 4 (1983),
  pp.~244--260.

\bibitem{BrennerScott}
{ S.~C. Brenner and L.~R. Scott}, {\em The Mathematical Theory of Finite
  Element Methods}, vol.~15 of Texts in Applied Mathematics, Springer-Verlag,
  New York, 1994.

\bibitem{Chatelin}
{ F.~Chatelin}, {\em Spectral Approximation of Linear Operators}, Computer
  Science and Applied Mathematics, Academic Press, Inc. [Harcourt Brace
  Jovanovich, Publishers], New York, 1983.
\newblock With a foreword by P. Henrici, With solutions to exercises by Mario
  Ahu\'{e}s.

\bibitem{ChenHeLiXie}
{ H.~Chen, Y.~He, Y.~Li, and H.~Xie}, {\em A multigrid method for eigenvalue
  problems based on shifted-inverse power technique}, Eur. J. Math., 1 (2015),
  pp.~207--228.

\bibitem{ChenXieXue}
{ H.~Chen, H.~Xie, and F.~Xu}, {\em A full multigrid method for eigenvalue
  problems}, J. Comput. Phys., 322 (2016), pp.~747--759.

\bibitem{ChenZou}
{ Z.~Chen and J.~Zou}, {\em The finite element method for elliptic equations
  with discontinuous coefficients}, Numerische Mathematik, 79 (1998).

\bibitem{Ciarlet}
{ P.~G. Ciarlet}, {\em The Finite Element Method for Elliptic Problems},
  vol.~4, North-Holland Publishing Co., Amsterdam-New York-Oxford, 1978.
\newblock Studies in Mathematics and its Applications.

\bibitem{DiLiTangZhang}
{ Y.~Di, R.~Li, T.~Tang, and P.~Zhang}, {\em Moving mesh finite element
  method for the incompressible Navier-Stokes equations}, SIAM J. Sci.
  Comput., 26 (2005).

\bibitem{hypre}
{ R.~D. Falgout and U.~M. Yang}, {\em hypre: A library of high performance
  preconditioners}, in Computational Science --- ICCS 2002, P.~M.~A. Sloot,
  A.~G. Hoekstra, C.~J.~K. Tan, and J.~J. Dongarra, eds., Berlin, Heidelberg,
  2002, Springer Berlin Heidelberg, pp.~632--641.

\bibitem{GongXieYan_SIAM}
{ W.~Gong, H.~Xie, and N.~Yan}, {\em A multilevel correction method for
  optimal controls of elliptic equations}, SIAM J. Sci. Comput., 37 (2015),
  pp.~A2198--A2221.

\bibitem{GongXieYan_JSC}
\leavevmode\vrule height 2pt depth -1.6pt width 23pt, {\em Adaptive multilevel
  correction method for finite element approximations of elliptic optimal
  control problems}, J. Sci. Comput., 72 (2017), pp.~820--841.

\bibitem{Hackbusch}
{ W.~Hackbusch}, {\em On the computation of approximate eigenvalues and
  eigenfunctions of elliptic operators by means of a multi-grid method}, SIAM
  J. Numer. Anal., 16 (1979), pp.~201--215.

\bibitem{Hackbusch_Book}
{ W.~Hackbusch}, {\em Multigrid methods and applications}, vol.~4 of
  Springer Series in Computational Mathematics, Springer-Verlag, Berlin, 1985.

\bibitem{hanYangBi}
{ J.~Han, Y.~Yang, and H.~Bi}, {\em A new multigrid finite element method
  for the transmission eigenvalue problems}, Appl. Math. Comput., 292 (2017),
  pp.~96--106.

\bibitem{HanLiXie}
{ X.~Han, Y.~Li, and H.~Xie}, {\em A multilevel correction method for
  {S}teklov eigenvalue problem by nonconforming finite element methods}, Numer.
  Math. Theory Methods Appl., 8 (2015), pp.~383--405.

\bibitem{HanLiXieYou}
{ X.~Han, Y.~Li, H.~Xie, and C.~You}, {\em Local and parallel finite element
  algorithm based on multilevel discretization for eigenvalue problems}, Int.
  J. Numer. Anal. Model., 13 (2016), pp.~73--89.

\bibitem{HanXieXu}
{ X.~Han, H.~Xie, and F.~Xu}, {\em A cascadic multigrid method for
  eigenvalue problem}, J. Comput. Math., 35 (2017), pp.~74--90.

\bibitem{FreeFEM}
{ F.~Hecht}, {\em New development in freefem++}, Journal of Numerical
  Mathematics.

\bibitem{HuXieXu}
{ G.~Hu, H.~Xie, and F.~Xu}, {\em A multilevel correction adaptive finite
  element method for {K}ohn-{S}ham equation}, J. Comput. Phys., 355 (2018),
  pp.~436--449.

\bibitem{JiSunXie}
{ X.~Ji, J.~Sun, and H.~Xie}, {\em A multigrid method for {H}elmholtz
  transmission eigenvalue problems}, J. Sci. Comput., 60 (2014), pp.~276--294.

\bibitem{JXXX}
{ S.~Jia, H.~Xie, M.~Xie, and F.~Xu}, {\em A full multigrid method for
  nonlinear eigenvalue problems}, Sci. China Math., 59 (2016), pp.~2037--2048.

\bibitem{FreeFEM_2}
{ P.~{Jolivet}, F.~{Hecht}, F.~{Nataf}, and C.~{Prud'homme}}, {\em Scalable
  domain decomposition preconditioners for heterogeneous elliptic problems}, in
  SC '13: Proceedings of the International Conference on High Performance
  Computing, Networking, Storage and Analysis, 2013, pp.~1--11.

\bibitem{Li201019}
{ J.~Li, J.~M. Melenk, B.~Wohlmuth, and J.~Zou}, {\em Optimal a priori
  estimates for higher order finite elements for elliptic interface problems},
  Applied Numerical Mathematics, 60 (2010), pp.~19--37.

\bibitem{LiTangZhang_1}
{ R.~Li, T.~Tang, and P.~Zhang}, {\em Moving mesh methods in multiple
  dimensions based on harmonic maps}, J. Comput. Phys., 170 (2001),
  pp.~562--588.

\bibitem{LiTangZhang_2}
\leavevmode\vrule height 2pt depth -1.6pt width 23pt, {\em A moving mesh finite
  element algorithm for singular problems in two and three space dimensions},
  J. Comput. Phys., 177 (2002), pp.~365--393.

\bibitem{LinXie_2010}
{ Q.~Lin and H.~Xie}, {\em An observation on the {A}ubin-{N}itsche lemma and
  its applications}, Math. Pract. Theory, 41 (2011), pp.~247--258.

\bibitem{LinXie_2012}
\leavevmode\vrule height 2pt depth -1.6pt width 23pt, {\em A multilevel
  correction type of adaptive finite element method for {S}teklov eigenvalue
  problems}, in Applications of Mathematics 2012, Acad. Sci. Czech Repub. Inst.
  Math., Prague, 2012, pp.~134--143.

\bibitem{LinXie}
\leavevmode\vrule height 2pt depth -1.6pt width 23pt, {\em A multi-level
  correction scheme for eigenvalue problems}, Math. Comp., 84 (2015),
  pp.~71--88.

\bibitem{LXX}
{ Q.~Lin, H.~Xie, and F.~Xu}, {\em Multilevel correction adaptive finite
  element method for semilinear elliptic equation}, Appl. Math., 60 (2015),
  pp.~527--550.

\bibitem{Miller_2}
{ K.~Miller}, {\em Moving finite element methods {II}}, SIAM J. Numer.
  Anal., 18 (1981), pp.~1033--1057.

\bibitem{Miller_1}
{ K.~Miller and M.~R. N.}, {\em Moving finite element methods {I}}, SIAM J.
  Numer. Anal., 18 (1981), pp.~1019--1032.

\bibitem{PengBiLiyang}
{ Z.~Peng, H.~Bi, H.~Li, and Y.~Yang}, {\em A multilevel correction method
  for convection-diffusion eigenvalue problems}, Math. Probl. Eng.,  (2015),
  pp.~Art. ID 904347, 10.

\bibitem{SLEPc}
{ J.~E. Roman, C.~Campos, E.~Romero, and A.~Tom\v{a}s}, {\em Slepc users
  manual--scalable library for eigenvalue problem computations}, Tech. Report
  3.14, Universitat Polit`ecnica de Valencia, Spain.

\bibitem{Saad}
{ Y.~Saad}, {\em Numerical Methods for Large Eigenvalue Problems}, vol.~66
  of Classics in Applied Mathematics, Society for Industrial and Applied
  Mathematics (SIAM), Philadelphia, PA, 2011.
\newblock Revised edition of the 1992 original [ 1177405].

\bibitem{ScottZhang}
{ L.~R. Scott and S.~Zhang}, {\em Higher-dimensional nonnested multigrid
  methods}, Math. Comp., 58 (1992), pp.~457--466.

\bibitem{shaidurov1995multigrid}
{ V.~V. Shaidurov}, {\em Multigrid Methods for Finite Elements}, vol.~318 of
  Mathematics and its Applications, Kluwer Academic Publishers Group,
  Dordrecht, 1995.
\newblock Translated from the 1989 Russian original by N. B. Urusova and
  revised by the author.

\bibitem{ToselliWidlund}
{ A.~Toselli and O.~Widlund}, {\em Domain Decomposition Methods---Algorithms
  and Theory}, vol.~34 of Springer Series in Computational Mathematics,
  Springer-Verlag, Berlin, 2005.

\bibitem{XiJiZhang}
{ Y.~Xi, X.~Ji, and S.~Zhang}, {\em A multi-level mixed element scheme of
  the two-dimensional {H}elmholtz transmission eigenvalue problem}, IMA J.
  Numer. Anal., 40 (2020), pp.~686--707.

\bibitem{Xie_JCP}
{ H.~Xie}, {\em A multigrid method for eigenvalue problem}, J. Comput.
  Phys., 274 (2014), pp.~550--561.

\bibitem{Xie_IMA}
\leavevmode\vrule height 2pt depth -1.6pt width 23pt, {\em A type of multilevel
  method for the {S}teklov eigenvalue problem}, IMA J. Numer. Anal., 34 (2014),
  pp.~592--608.

\bibitem{Xie_Nonlinear}
\leavevmode\vrule height 2pt depth -1.6pt width 23pt, {\em A multigrid method
  for nonlinear eigenvalue problems}, Sci. Sin. Math., 45 (2015),
  pp.~1193--1204.

\bibitem{Xie_BIT}
\leavevmode\vrule height 2pt depth -1.6pt width 23pt, {\em A type of
  multi-level correction scheme for eigenvalue problems by nonconforming finite
  element methods}, BIT, 55 (2015), pp.~1243--1266.

\bibitem{XieWu}
{ H.~Xie and X.~Wu}, {\em A multilevel correction method for interior
  transmission eigenvalue problem}, J. Sci. Comput., 72 (2017), pp.~586--604.

\bibitem{XieXie}
{ H.~Xie and M.~Xie}, {\em A multigrid method for ground state solution of
  {B}ose-{E}instein condensates}, Commun. Comput. Phys., 19 (2016),
  pp.~648--662.

\bibitem{XieXieZhang}
{ H.~Xie, M.~Xie, and N.~Zhang}, {\em An efficient multigrid method for
  semilinear elliptic equation}, J. Num. Method. Comp. Appl., 40 (2019),
  pp.~143--160.

\bibitem{XieZhangOwhadi}
{ H.~Xie, L.~Zhang, and H.~Owhadi}, {\em Fast eigenpairs computation with
  operator adapted wavelets and hierarchical subspace correction}, SIAM J.
  Numer. Anal., 57 (2019), pp.~2519--2550.

\bibitem{XieZhou}
{ H.~Xie and T.~Zhou}, {\em A multilevel finite element method for
  {F}redholm integral eigenvalue problems}, J. Comput. Phys., 303 (2015),
  pp.~173--184.

\bibitem{XuXie}
{ F.~Xu and H.~Xie}, {\em A full multigrid method for semilinear elliptic
  equation}, Appl. Math., 62 (2017), pp.~225--241.

\bibitem{XuXieZhang}
{ F.~Xu, H.~Xie, and N.~Zhang}, {\em An eigenwise parallel augmented
  subspace method for eigenvalue problems}, arXiv: 1908.10251,  (2019).

\bibitem{Xu}
{ J.~Xu}, {\em Iterative methods by space decomposition and subspace
  correction}, SIAM Rev., 34 (1992), pp.~581--613.

\bibitem{Xu_Nonsymmetric}
\leavevmode\vrule height 2pt depth -1.6pt width 23pt, {\em A new class of
  iterative methods for nonselfadjoint or indefinite problems}, SIAM J. Numer.
  Anal., 29 (1992), pp.~303--319.

\bibitem{XuXiaowen}
{ X.~Xu}, {\em Parallel algebraic multigrid methods: state-of-the art and
  challenges for extreme-scale applications}, J. Num. Method. Comp. Appl., 40
  (2019), pp.~243--260.

\bibitem{YueXieXie}
{ M.~Yue, H.~Xie, and M.~Xie}, {\em A cascadic multigrid method for
  nonsymmetric eigenvalue problem}, Appl. Numer. Math., 146 (2019), pp.~55--72.

\bibitem{ZhangHanHeXieYou}
{ N.~Zhang, X.~Han, Y.~He, H.~Xie, and C.~You}, {\em An algebraic multigrid
  method for eigenvalue problems in some different cases}, arXiv: 1503.08462,
  (2015).

\bibitem{ZhangXuXie}
{ N.~Zhang, F.~Xu, and H.~Xie}, {\em An efficient multigrid method for
  ground state solution of {B}ose-{E}instein condensates}, Int. J. Numer. Anal.
  Model., 16 (2019), pp.~789--803.

\bibitem{ZhangXiJi}
{ S.~Zhang, Y.~Xi, and X.~Ji}, {\em A multi-level mixed element method for
  the eigenvalue problem of biharmonic equation}, J. Sci. Comput., 75 (2018),
  pp.~1415--1444.

\end{thebibliography}

\end{document}